\documentclass[]{article}
\usepackage{amsfonts}
\usepackage{amsmath,amsthm}
\usepackage{amssymb}
\usepackage{color,soul}
\usepackage{textcomp}
\usepackage{geometry}
\usepackage{enumitem}
\usepackage[utf8]{inputenc}
\usepackage{todonotes}

\usepackage{graphicx}
\graphicspath{{figures/}}
\usepackage{subcaption}
\usepackage{epstopdf}
\usepackage{float}
\usepackage{tikz}
\usetikzlibrary{positioning}
\usetikzlibrary{arrows}
\usetikzlibrary{arrows.meta}
\usetikzlibrary{calc}
\newtheorem{theorem}{Theorem}[section]
\newtheorem{prop}{Proposition}
\newtheorem{lemm}{Lemma}
\newtheorem{remark}{Remark}

\newcommand{\cO}{\mathcal{O}}

\newcommand{\cC}{\mathcal{C}}
\newcommand{\e}{\epsilon}

\title{A geometric analysis of the SIR, SIRS and SIRWS epidemiological models}
\author{Hildeberto Jard\'on-Kojakhmetov{$^{1}$}, Christian Kuehn{$^{1}$},\\
	Andrea Pugliese{$^{2}$}, Mattia Sensi{$^{2}$}\\[1em]
	$^1$Technische Universit{\"a}t M{\"u}nchen (TUM) \\
	$^2$Universit\`a degli Studi di Trento \\ }

\begin{document}

\maketitle

\begin{abstract}
We study fast-slow versions of the SIR, SIRS, and SIRWS epidemiological models. The multiple time scale behavior is introduced to account for large differences between some of the rates of the epidemiological pathways. Our main purpose is to show that the fast-slow models, even though in nonstandard form, can be studied by means of Geometric Singular Perturbation Theory (GSPT). In particular, without using Lyapunov's method, we are able to not only analyze the stability of the endemic equilibria but also to show that in some of the models limit cycles arise.  We show that the proposed approach is particularly useful in more complicated (higher dimensional) models such as the SIRWS model, for which we provide a detailed description of its dynamics by combining analytic and numerical techniques.
	\end{abstract}


\section{Introduction}

Epidemic modelling has grown from the  pioneering 1927 article by Kermack and McKendrick \cite{KeMcK} into a wide body of theory  and applications to several diseases \cite{AndMay_bk,Heth,keeling_rohani,DiekHeesBrit,Martcheva_bk}, used also for developing appropriate control strategies.

The model by Kermack and McKendrick \cite{KeMcK} was of S-I-R type, meaning that individuals are classified as  \emph{Susceptibles} ($S$), \emph{Infected} ($I$) or \emph{Recovered} ($R$), and that the only possible transitions are $S \to I$ (new infection) and $I \to R$ (recovery with permanent immunity). As that model does not consider new births or deaths (other than because of the disease), it is appropriate for an epidemic that develops on a time-scale much faster than demographic turn-around. 
The epidemic SIR model was extended by Soper who added \cite{Soper} (constant) birth and death rates to the model, obtaining the so-called SIR endemic model, that has been extensively analysed in the following decades, especially to investigate how to explain the apparent periodicities in the notifications of childhood diseases \cite{Smi83,Kee01}. 
The SIR endemic model can be seen as the basis, over which more complex and realistic models have been built.

The difference in time-scales between  epidemic spread and demographic turnaround has been observed by several authors. Smith \cite{Smi83} introduced a small parameter $\e$ as the ratio between the average lengths of the infection period and of life; he proved that, if the contact rate is a sinusoidal function of period 1  and $\e$ is sufficiently small, a subharmonic bifurcation of a 2-periodic stable positive solution can occur. Andreasen \cite{And93} showed that, for $\e$ small enough, the endemic equilibrium is always stable in a certain class of age-dependent SIR models. Diekmann, Heesterbeek and Britton \cite{DiekHeesBrit} have exploited the fact that $\e$ is  a small parameter in an informal argument about the minimum community size in which a measles-like infection can persist. However, to our knowledge very few authors have systematically used  geometric singular perturbation theory as a tool to investigate properties of epidemic models.
We only know of the paper by Rocha \textit{et al.} \cite{Rocha} that used singular perturbation methods for the analysis of a SIRUV model for a  vector-borne epidemic.
 
Our main objective in this paper is to show that under certain assumptions of the system parameters (namely the transition rates between states), tools from Geometric Singular Perturbation Theory (GSPT) are suitable to describe the intricate dynamics that such models exhibit due to the presence of multiple time scales. 

The first part of the paper is devoted to the classical SIR and SIRS epidemic models, that we analyse in the limiting case of $\e \to 0$. For such models, it is well known that, when $R_0 > 1$, there exists a unique endemic equilibrium, which is globally asymptotically stable.\\

In the second part, we instead consider a model, named SIRWS, introduced for pertussis in  \cite{Lavi}, and partially analysed in  \cite{Dafi}. In the model it is assumed that immunity wanes in two stages: after recovering from infection individuals are totally immune, but then immune memory starts to fade: if they are challenged by the pathogen when they are in the stage of partial immunity, they recover a complete immunity; otherwise, they completely lose immunity, and re-enter the susceptible stage.
 
Our main results can be summarized as follows:
\begin{itemize}[leftmargin=*]
	\item For the fast-slow SIR and SIRS models we capture the transient behaviour from an initial introduction of the infection, and show that, when $R_0 > 1$, the dynamics leads, in the slow time-scale, to  a neighbourhood of the endemic equilibrium, see Sections \ref{SIR} and \ref{SIRS}. Then convergence to the equilibrium can be established by local methods.
	\item For the fast-slow SIRWS model, in particular,  we confirm the result obtained numerically in \cite{Dafi} that stable periodic epidemic outburst can exist. Moreover, we give a detailed description of the system parameters for which such behaviour occurs and the corresponding time scales involved, see Section \ref{sec:SIRWS}.
\end{itemize}

Our mathematical analysis is largely based on GSPT, see more details in Section \ref{sec:preliminaries}.

In such a context, it is worth mentioning that the models we study are not immediately, nor globally, in a standard singularly perturbed form, but in each model the fast-slow decomposition appears only in specific regions of the phase space, similarly to what is considered in e.g. \cite{KoSz,KuSz}. As it is usually the case in such biological models, the main difficulty for analysis is due to the loss of normal hyperbolicity of the critical manifold. To overcome this obstacle, we use here the so called entry-exit function, as presented by De Maesschalck and Schecter \cite{DeMaSch}, which gives details regarding the behaviour of an orbit in regions where the critical manifold changes its stability properties. Moreover, for the modified SIRWS system we present a combination of analytical and numerical studies regarding the dependence of the dynamics with respect to some of the parameters, and compare our results with the ones obtained in \cite{Dafi}. In particular, we focus on the interplay between life expectancy (or birth/death rate) and boosting rate, and on how different values of these parameters can give rise to damped or sustained oscillations. Finally, the novelty of our analysis is not confined to the usage of GSPT in the context of the well-known SIR model, but we also show that our techniques can be potentially used in higher dimensional systems (as the SIRWS model). This is rather important since the well-studied SIR and SIRS models often depend on Lyapunov's method to show stability of trajectories \cite{Lyap}, and it is known that Lyapunov functions are difficult to obtain. Our GSPT analysis does not require global Lyapunov functions.

The remainder of this paper is arranged as follows: in Section \ref{sec:preliminaries} we provide some necessary mathematical preliminaries which will be later used for the analysis of the models. Afterwards, we present in Section \ref{sec:models} the mathematical analysis of the SIR, SIRS, and the SIRWS epidemiological models. We finish in Section \ref{sec:conclusions} with a summary and an outlook of open-problems regarding modelling and analysis of epidemiological models with fast-slow dynamics.

\section{Preliminaries}\label{sec:preliminaries}

In the main part of this paper we study three compartment models whose dynamics evolve at distinct time scales. Therefore, we now provide a brief description of Geometric Singular Perturbation Theory (GSPT), and in particular of the entry-exit function \cite{DeMaSch}, which is fundamental in our analysis.

\subsection{Fast-slow systems}

The term ``fast-slow systems'' is commonly used to model phenomena which evolve on two (or more) different time scales \cite{bertram2017multi,Kuehn}. Often such behaviour can be described by a singularly perturbed ordinary differential equation (ODE), that is
\begin{equation}\label{eq:fen-two}
\begin{aligned}
\epsilon \dot{x}&=f(x,y,\epsilon),\\
\dot{y}&=g(x,y,\epsilon),
\end{aligned}
\end{equation}
where $x=x(\tau) \in \mathbb{R}^m$, $y=y(\tau) \in \mathbb{R}^n$, with $m,n\geq 1$, are the fast and slow variables respectively, $f$ and $g$ are functions of class $\mathcal{C}^k$, with $k$ as large as needed, and $0<\epsilon\ll 1$ is a small parameter which gives the ratio of the two time scales. Here the overdot ( $\dot{}$ ) indicates $\frac{\textnormal{d}}{\textnormal{d}\tau}$. The system~(\ref{eq:fen-two}) is formulated on the \textit{slow time} scale $\tau$. When studying fast-slow systems we often define a new \textit{fast time} $t=\tau/\epsilon$ with which \eqref{eq:fen-two} can be rewritten as
\begin{equation}\label{eq:fen-one}
\begin{aligned}
x'&=f(x,y,\epsilon),\\
y'&=\epsilon g(x,y,\epsilon),
\end{aligned}
\end{equation}%
where now the prime ( $'$ ) indicates $\frac{\textnormal{d}}{\textnormal{d}t}$. Clearly, since we simply rescaled the time variable, systems~(\ref{eq:fen-two}) and~(\ref{eq:fen-one}) are equivalent for $\epsilon> 0$.

Fast-slow systems given by \eqref{eq:fen-two}-\eqref{eq:fen-one} are said to be \emph{in standard form}. In a more general context, it is possible to have a fast-slow system given by
\begin{equation}
	 z' = F(z,\epsilon),
\end{equation}
where the time scale separation is not explicit. In fact, many biological models \cite{KoSz,KuSz}, among others, and in particular the models we study in this paper are in such non-standard form.

The main idea of GSPT is to consider \eqref{eq:fen-two}-\eqref{eq:fen-one} in the limit $\epsilon\to0$ and then use perturbation arguments to describe the dynamics of the full fast-slow system. The motivation behind this strategy is that one expects that the analysis of the limit systems ($\epsilon=0$) is simpler compared to the analysis of \eqref{eq:fen-two}-\eqref{eq:fen-one} with $\epsilon>0$.

Taking the limit $\epsilon\rightarrow 0$ in systems~(\ref{eq:fen-two}) and~(\ref{eq:fen-one}) yields, respectively
\begin{equation}\label{eq:fen-slow}
\begin{aligned}
0&=f(x,y,0),\\
\dot{y}&=g(x,y,0),
\end{aligned}
\end{equation}%
and
\begin{equation}\label{eq:fen-fast}
\begin{aligned}
x'&=f(x,y,0),\\
y'&=0,
\end{aligned}
\end{equation}
where \eqref{eq:fen-slow} is called \textit{reduced subsystem} (or \textit{slow subsystem}), and \eqref{eq:fen-fast} is called \emph{the layer equation} (or \textit{fast subsystem}). We note that the reduced subsystem describes a dynamic evolution constrained to the set
$$\mathcal{C}_0=\{ x \in \mathbb{R}^m, y \in \mathbb{R}^n \,|\, f(x,y,0)=0\},$$ 
which is called the \textit{critical manifold}. On the other hand, we note that $\mathcal{C}_0$ defines the set of equilibrium points of the layer equation.

Fenichel's theorems, which are the basis of GSPT, require certain assumptions on $\mathcal{C}_0$. Namely, we suppose there exists an $n$-dimensional compact submanifold $\mathcal{M}_0$, possibly with boundary, contained in $\mathcal{C}_0$. Moreover, the manifold $\mathcal{M}_0$ is assumed to be \textit{normally hyperbolic} and \textit{locally invariant}, which mean, respectively, that the eigenvalues of the Jacobian $\textnormal{D}_x f(x,y,0)|_{\mathcal{M}_0}$ are uniformly bounded away from
the imaginary axis, and that the flow can only leave $\mathcal{M}_0$ through its boundary. In such a setting, the following can be proved (see \cite{Feni1}):
\begin{theorem}
\label{FenTh1}
For $\epsilon>0$ sufficiently small, there exists a manifold $\mathcal{M}_{\epsilon}$, called \textnormal{slow manifold}, which lies $\mathcal{O}(\epsilon)$ close to $\mathcal{M}_0$, is diffeomorphic to $\mathcal{M}_0$ and is locally invariant under the flow of~(\ref{eq:fen-one}).
\end{theorem}
We note that the manifold $\mathcal{M}_{\epsilon}$ is usually not unique, but all the possible choices lie $\mathcal{O}(\e^{-K/\epsilon})$-close to each other, for some $K>0$. Therefore, in most cases the choice of slow manifold $\mathcal{M}_{\epsilon}$ does not change the analytical and numerical results.

With the usual definitions for stable and unstable manifolds (see, for example, equations (6.3) in \cite{Kuehn})
\begin{align*}
W^{\textnormal{s}}(\mathcal{M}_0)=\{(x,y): \phi_t(x,y)\rightarrow \mathcal{M}_0 \hbox{ as } t\rightarrow +\infty \}, \\
W^{\textnormal{u}}(\mathcal{M}_0)=\{(x,y): \phi_t(x,y)\rightarrow \mathcal{M}_0 \hbox{ as } t\rightarrow -\infty \},
\end{align*}
where $\phi_t$ denotes the flow of system~(\ref{eq:fen-fast}), Fenichel's second theorem ensures that $W^{\textnormal{s}}(\mathcal{M}_0)$ and $W^{\textnormal{u}}(\mathcal{M}_0)$ persist under perturbation as well:
\begin{theorem}
\label{FenTh2}
For $\epsilon>0$ sufficiently small, there exist manifolds $W^{\textnormal{s}}(\mathcal{M}_{\epsilon})$ and $W^{\textnormal{u}}(\mathcal{M}_{\epsilon})$ which lie $\mathcal{O}(\epsilon)$ close to and are diffeomorphic to $W^{\textnormal{s}}(\mathcal{M}_0)$ and $W^{\textnormal{u}}(\mathcal{M}_0)$ respectively, and are locally invariant under the flow of~(\ref{eq:fen-one}).
\end{theorem}
In practical terms, Fenichel's theorems show that for $\epsilon>0$ sufficiently small, the dynamics of \eqref{eq:fen-two}-\eqref{eq:fen-one} are a regular perturbation of the limit dynamics \eqref{eq:fen-slow}-\eqref{eq:fen-fast} within a small neighbourhood of the critical manifold.

When the manifold $\mathcal{M}_0$ is not normally hyperbolic, some more advanced tools, such as the \textit{blow-up method} (see \cite{HildeKu}), may need to be invoked. All of the systems we analyse below have one non-hyperbolic point in the biologically relevant region. Thus, in order to describe the relevant dynamics we need to use extra techniques besides Fenichel's theorems. Due to the properties of the models to be studied, it turns out that the \emph{entry-exit function} \cite{Demaaa,DeMaSch} is suitable.

\subsection{Entry-exit function} \label{entryexit}

The entry-exit function gives, in the form of a Poincar\'e map between two sections in phase space, an estimate of the behaviour of the orbits near the point in which the critical manifold changes stability (from attracting to repelling), in a class of singularly perturbed systems. Intuitively, the result can be interpreted as a ``build up'' of repulsion near the repelling part of the slow manifold, which needs to compensate the attraction which was built up near the attracting part before the orbit can leave an $\mathcal{O}(\epsilon)$ neighbourhood of the critical manifold.

More specifically, this construction applies to systems of the form 
\begin{equation}\label{eq:entex}
\begin{aligned}
x'&= f(x,y,\epsilon)x,\\
y'&=\epsilon g(x,y,\epsilon),
\end{aligned}
\end{equation}
with $(x,y)\in \mathbb{R}^2$, $g(0,y,0)>0$ and $\textnormal{sign}(f(0,y,0))=\textnormal{sign}(y)$. Note that for $\epsilon=0$, the $y$-axis consists of normally attracting/repelling equilibria if $y$ is negative/positive, respectively.
\begin{figure}[htbp]\centering
	\begin{tikzpicture}
		\node at (0,0){\includegraphics[scale=0.85]{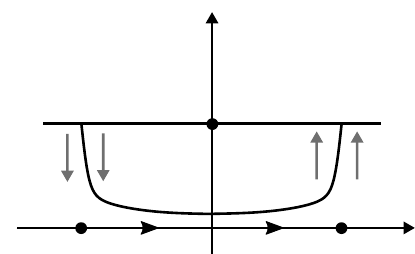}};
		\node at (2,-.9){$y$};
		\node at (0.05,1.15){$x$};
		\node at (-1.1,.2){$x=x_0$};
		\node at (-1,-1.15) {$y_0$};
		\node at ( 1.1,-1.15) {$p_\epsilon(y_0)$};
	\end{tikzpicture}
	\caption{Visualization of the entry-exit map on the line $x=x_0$}
	\label{fig:entrex}
\end{figure}
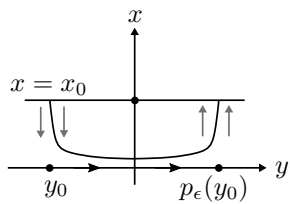\\

Consider a horizontal line $\{x=x_0\}$, which is $\mathcal{O}(\epsilon)$-close to the $y$-axis. An orbit of \eqref{eq:entex} that intersects such a line at $y=y_0<0$ (entry) re-intersects it again (exit) at $y=p_\epsilon(y_0)$, as sketched in Figure~\ref{fig:entrex}. De Maesschalck \cite{Demaaa} shows that, as $\epsilon \rightarrow 0$, the image of the return map $p_\epsilon(y_0)$ to the horizontal line $x=x_0$ approaches $p_0(y_0)$ given implicitly by
\begin{equation}\label{eq:pzero}
\int_{y_0}^{p_0(y_0)} \frac{f(0,y,0)}{g(0,y,0)}\textnormal{d}y = 0.
\end{equation}
In the following sections, the entry-exit function $p_0$ plays a crucial role in the analysis of three different epidemiological models. In particular, the analysis of the SIRWS model relies on a multi-dimensional version of the entry-exit map, provided in a recent paper by Hsu and Ruan \cite{HsuRuan}.

\section{Analysis of the SIR, SIRS and SIRWS models}\label{sec:models}

In this section we analyse three different epidemiological models, giving a short interpretation of the equations and then proceeding to use the techniques of GSPT,  especially the entry-exit function, to deduce information about the behaviour of each one.

\subsection{SIR model}\label{SIR}

We consider a SIR compartment model (presented in a similar form in \cite{Heth} and with the same underlying dynamics in \cite{KeMcK}) as depicted in Figure \ref{fig:SIR} and with corresponding equations given as in \eqref{eq:SIR1}

\begin{minipage}{.4\textwidth}\centering
		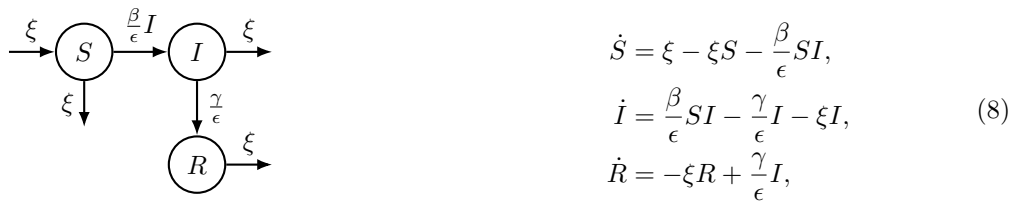
\begin{figure}[H]
			\centering
		\begin{tikzpicture}
		\node[draw,circle,thick,minimum size=.75cm] (s) at (1,0) {$S$};
		\node[draw,circle,thick,minimum size=.75cm] (i) at (2.5,0) {$I$};
		\node[draw,circle,thick,minimum size=.75cm] (r) at (2.5,-1.5) {$R$};
		\draw[-{Latex[length=2.mm, width=1.5mm]},thick] (s)--(i) node[above, midway]{$\frac{\beta}{\epsilon}I$};
		\draw[-{Latex[length=2.mm, width=1.5mm]},thick] (i)--(r) node[right, midway]{$\frac{\gamma}{\epsilon}$};
		\draw[{Latex[length=2.mm, width=1.5mm]}-,thick] (s)--++(-1,0) node[above,midway]{$\xi$};
		\draw[-{Latex[length=2.mm, width=1.5mm]},thick] (s)--++(0,-1) node[left,midway]{$\xi$};
		\draw[-{Latex[length=2.mm, width=1.5mm]},thick] (i)--++(1,0) node[above,midway]{$\xi$};
		\draw[-{Latex[length=2.mm, width=1.5mm]},thick] (r)--++(1,0) node[above,midway]{$\xi$};
		\end{tikzpicture}
			\caption{Flow diagram for~(\ref{eq:SIR1}).}
			\label{fig:SIR}
		\end{figure}
\end{minipage}
\hfill
\begin{minipage}{.5\textwidth}
	\begin{equation}\label{eq:SIR1}
	\begin{aligned}
	\dot{S}&=\xi -\xi S -\frac{\beta}{\epsilon}SI,\\
	\dot{I}&=\frac{\beta}{\epsilon}SI-\frac{\gamma}{\epsilon}I-\xi I,\\
	\dot{R}&=-\xi R+\frac{\gamma}{\epsilon}I,
	\end{aligned}
	\end{equation}
\end{minipage}\\
\\
where $S(\tau)$, $I(\tau)$, $R(\tau)$ denote the susceptible, infected and recovered proportion of the population respectively. Since the $(S,I,R)$ variables represent fractions of a population, they are assumed to be non-negative for all $\tau\geq0$. Observe that the non-negative octant of $\mathbb{R}^3$, to be denoted by $\mathbb R^3_{\geq0}$, and in particular the set $\left\{ (S,I,R)\in\mathbb R^3_{\geq0}\, | \, \, 0\leq S+I+R\leq 1 \right\}$, are invariant under the flow of~(\ref{eq:SIR1}). 

The parameter $\xi$ in \eqref{eq:SIR1} refers to the birth rate and is assumed to be equal to the death rate. Furthermore, as depicted in Figure \ref{fig:SIR}, we also assume that all individuals are born susceptible. Similarly, the parameter $\beta$ and $\gamma$ refer, respectively, to the rates at which susceptible individuals are infected and the latter are recovered. In our analysis the parameters $\xi$, $\beta$ and $\gamma$ are of order $\mathcal{O}(1)$. Note that we introduce a small positive parameter $0< \epsilon \ll 1$, which gives rise to the difference in magnitude between the large infection rate $\beta /\epsilon$, the large recovery rate $\gamma/\epsilon$ and the birth/death rate. Such a difference represents a highly contagious disease with a short infection period.

As stated above, $S(\tau)$, $I(\tau)$ and $R(\tau)$ represent proportions of the population. Consistently the plane $\{S+I+R = 1\}$ is invariant for system  \eqref{eq:SIR1} . Hence, we can assume $R=1-S-I$, which allows us to reduce~(\ref{eq:SIR1}) to
\begin{equation}\label{eq:SIR1red1}
\begin{aligned}
\dot{S}&=\xi-\xi S -\frac{\beta}{\epsilon}SI,\\
\dot{I}&=\frac{\beta}{\epsilon}SI-\frac{\gamma}{\epsilon}I-\xi I.
\end{aligned}
\end{equation}
By rescaling time, system~(\ref{eq:SIR1red1}) can also be written as
\begin{equation}\label{eq:SIR1red3}
\begin{aligned}
S'&=\epsilon\xi(1- S) -\beta SI,\\
I'&=I(\beta S-\gamma-\epsilon\xi).
\end{aligned}
\end{equation}

Note that system~(\ref{eq:SIR1red3}) is a fast-slow system in non-standard form, as it often occurs in biological models \cite{KoSz,KuSz}. Later we perform a convenient rescaling that brings \eqref{eq:SIR1red3} into a standard form.

The corresponding critical manifold is the set $\mathcal{C}_0=\{(S,I)\in \mathbb{R}^2 \, | \, I=0\}$, and the slow flow along it is given by $\dot{S}=\xi(1-S)$, which implies flow towards the point $S=1$.  In the $\epsilon\rightarrow 0$ limit, we recover from~(\ref{eq:SIR1red3}) the basic dynamics for the $(S,I)$ couple in a standard SIR system (see \cite{Heth}), namely
\begin{equation}\label{eq:SIR1redlim}
\begin{aligned}
S'&=-\beta SI,\\
I'&=I(\beta S-\gamma).
\end{aligned}
\end{equation}

In particular, it follows from linearization of \eqref{eq:SIR1redlim} along $\mathcal{C}_0$ that the critical manifold is attracting for $S<\frac{\gamma}{\beta}$, repelling for $S>\frac{\gamma}{\beta}$, and loses normal hyperbolicity at  $S=\frac{\gamma}{\beta}$.

From here on, we assume the basic reproduction number to be $R_0=\beta /\gamma >1$. This means that the disease is able to spread through the population. In particular, 
as stated in the well known next Lemma \cite{heth05,KeMcK}, the previous assumption implies that, for every initial condition $S(0)=S_0 > 1/R_0$, there exists a unique 
$S_\infty < 1/R_0$ such that a trajectory of~(\ref{eq:SIR1redlim}) with initial conditions $(S_0,I_0)$ converges towards $(S_\infty,0)$ as $t \rightarrow + \infty$.
\begin{lemm}\label{lemm:Sinf}$\Gamma(S,I)=\gamma \ln(S)-\beta (S+I)$ is a constant of motion for system~(\ref{eq:SIR1redlim}), and all its orbits in the first quadrant 
are heteroclinic to two points on the $S$-axis.
\end{lemm}

\begin{figure}[htbp]\centering

		\begin{tikzpicture}
		\node at (0,0){\includegraphics[scale=1.4]{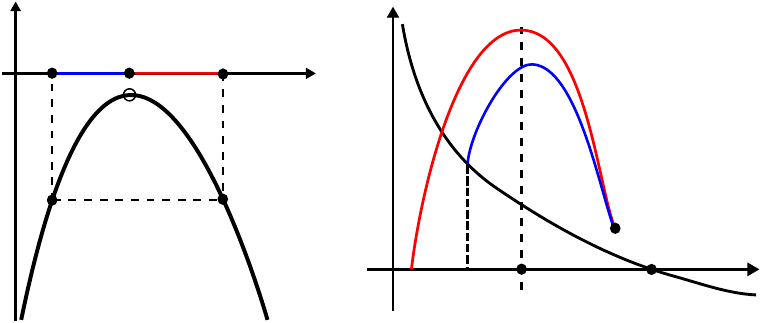}};
		\node at (0.2,2.4){$I$};
		\node at (-2.2,1.5){$1$};
		\node at (-.8,1.25){$S$};
		\node[text=red] at (-2.9,1.5){$S_0$};
		\node at (-3.6,1.6){$\frac{1}{R_0}$};
		\node[text=blue] at (-4.2,1.5){$S_\infty$};
		\node at (-2.5,-1.5){$\Gamma(S,0)$};
		\node at (5.6,-1.5){$S$};
		\node at (4.0,-1.3){$1$};
		\node at (4.0,-0.8){$(S_0,I_0)$};
		\node at (6.3,-2.){$I=\frac{\epsilon\xi(1-S)}{\beta S}$};
		\node at (2.25,-2.2){$S=\frac{1}{R_0}$}; 
		\node[text=red] at (0.5,-1.9){$S_\infty$};
		\node at (0.9,-1.25){\footnotesize $\mathcal{O}(\epsilon)$ };
		\node at (0.9,-1.47){\footnotesize $\overbrace{}$};
		\end{tikzpicture}
		\caption{Left: function $\Gamma(S,0)$, intersection with horizontal lines give the starting and ending points of a heteroclinic orbit of the layer equation \eqref{eq:SIR1redlim}. Right: qualitative comparison between perturbed and unperturbed SIR systems in fast time scale. In red we show an orbit of \eqref{eq:SIR1redlim} given by $\Gamma(S,I)=\Gamma(S_0,I_0)$ and in blue a small perturbation of it corresponding the related orbit of \eqref{eq:SIR1red3}.}
				\label{fig:gamma}

\end{figure}
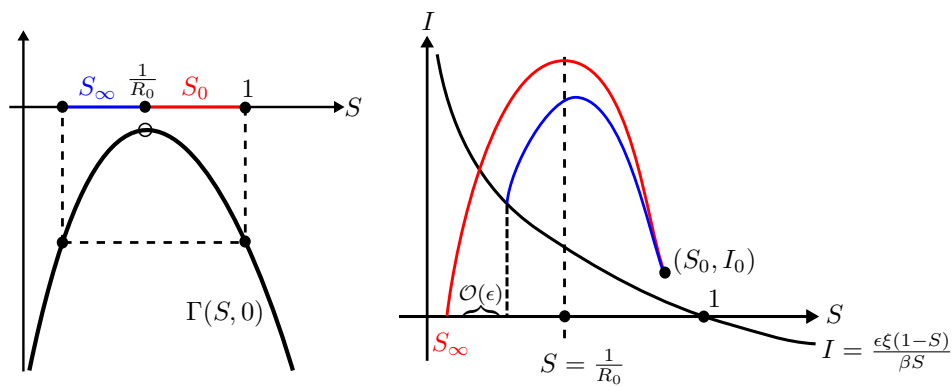

From Lemma \ref{lemm:Sinf} we define $S_\infty\in(0,\frac{1}{R_0})$ to be the unique non-trivial solution of the equation $\Gamma(S,0)=\Gamma(S_0,0)$ where $S_0>\frac{1}{R_0}$.

For future use, let us define the map
\begin{equation}
\label{Pi1}
\Pi_1:\{S\in (1/R_0,1] \}\rightarrow \{S\in (0,1/R_0)\}
\end{equation}
 that maps $S_0$ into $S_\infty$, and which is induced by the flow of \eqref{eq:SIR1redlim}, or is equivalently given by $\Gamma$.

So far, we know that the solutions of \eqref{eq:SIR1red3} away from the critical manifold are closely given by $\Gamma(S,I)$ as shown in the right side of Figure \ref{fig:gamma}. Therefore, the next step is to focus on a small region close to $\mathcal{C}_0$. That is, for the analysis that follows, we assume $I$ to be $\cO(\epsilon)$-small.
In particular, and following Lemma \ref{lemm:Sinf}, if we choose $I_0\in\mathcal{O}(\epsilon^2)$, we have an explicit relation (up to a $\mathcal{O}(\epsilon)$ error) between $S_\infty$ and $S_0$, namely, $\Gamma(S_\infty,0)\approx\Gamma(S_0,I_0)=\Gamma(S_0,0)+\cO(\epsilon)$.

Considering the signs of the derivatives in the perturbed system \eqref{eq:SIR1red3}, we see that orbits spiral counterclockwise. Moreover, system~\eqref{eq:SIR1red3} has a two equilibria, namely $(S,I)=(1,0)$ and one which is $\mathcal{O}(\epsilon)$-close to the point $(1/R_0 , 0)$, as shown in Figure~\ref{fig:orbita}, given by $(S,I)=(S_E,I_E):=(\frac{1}{R_0}+\epsilon\frac{\xi}{\beta}, \alpha_\epsilon (S_E))$, where
\begin{equation}\label{eq:alphaeps}
\alpha_\epsilon(S)=\frac{\epsilon\xi(1-S)}{\beta S}
\end{equation}
is obtained from the nullcline for $S$ in~(\ref{eq:SIR1}). Regular perturbation arguments imply that an orbit of the perturbed system  \eqref{eq:SIR1red3}, starting from a point $(S_0,I_0)$ with $I_0\in \mathcal{O}(\epsilon)$ and $S_0>S_E$, follows $\mathcal{O}(\epsilon)$-closely from below, since the $\mathcal{O}(\epsilon)$ contribution is negative, a power level of $\Gamma(S,I)$, until it reaches the nullcline of $S$ given by $I=\frac{\epsilon\xi(1-S)}{\beta S}$, as shown on the right half of Figure~\ref{fig:gamma}, at a point with $S$ coordinate $\mathcal{O}(\epsilon)$-close to $S_\infty$.
\begin{figure}[htbp]\centering
	\begin{tikzpicture}
		\node at (0,0){\includegraphics[scale=1.8]{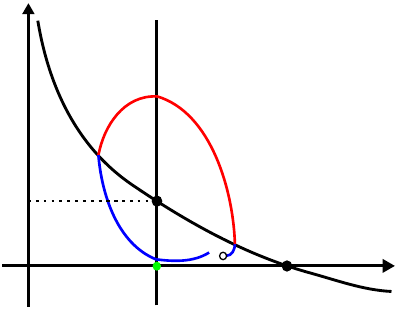}};
		\node at (-3.4,2.7){$I$};
		\node at (-0.1,-0.6){\small $(S_E,I_E)$};
		\node at (1.8,-1.8){$1$};
		\node at (3.5,-1.7) {$S$};
		\node at (3,-2.8) {$I=\alpha_\epsilon(S)$};
		\node[text=red] at (0.3,2.3) {$\dot{S}<0,\dot{I}>0$};
		\node[text=red] at (-1.85,2.3) {$\dot{S}<0,\dot{I}<0$};
		\node[text=blue] at (0.4,-2.5) {$S'<0,I'<0$};
		\node[text=blue] at (-1.95,-2.5) {$S'>0,I'>0$};
		\node at (-3.8,-1.5){{\large$\mathcal{O}(\epsilon)$} {\huge $\{$ }};
	\end{tikzpicture}
	\caption{Schematic representation of the orbits of \eqref{eq:SIR1red3} on the two time scales. Red: fast orbit; blue: slow orbit; green: non-hyperbolic point.}
			\label{fig:orbita}
\end{figure}
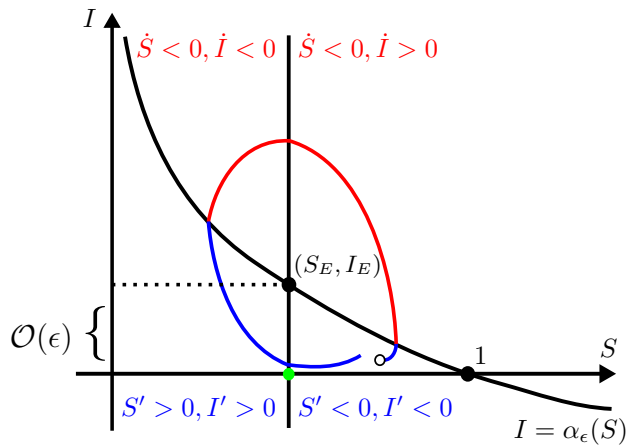

It is also well known \cite{heth05,Lyap} that the endemic equilibrium $(S_E,I_E)$ is globally asymptotically stable, as stated below.
\begin{theorem}
\label{theor:convergence} Consider \eqref{eq:SIR1red3}. All trajectories with initial conditions $0\leq S(0)\leq1$, $0<I(0)\leq1$ with $S(0)+I(0)\leq 1$ converge asymptotically towards the (endemic) equilibrium point $(S_E,I_E)$.
\end{theorem}
The theorem can be proved using the Lyapunov function
	\begin{equation}
		L_1(S,I)=S+I-S_E \ln (S) - I_E \ln(I)-C_E,
	\end{equation}
	with $C_E=S_E+I_E-S_E \ln (S_E) - I_E \ln(I_E)$, together with  Lasalle's invariance principle \cite{Lasalle}; or with \cite{Lyap,ShuaivdD}
	\begin{equation}
\label{Lyap2}
L_2(S,I) = I - I_E - I_E \ln(I/I_E) +\frac{\beta}{2(2\mu+\gamma) }(S + I - S_E + I_E)^2.
\end{equation}

Here we are going to describe how solutions approach the equilibrium, for $\e > 0$ small. Once it is shown that solutions are in a neighbourhood of the equilibrium, local methods can be used to prove convergence to the equilibrium. Such an approach will be used for the other models as well. Our motivation is to present a method of analysis that does not depend on finding a Lyapunov function, which is, in general, a difficult task.

A convenient step, which is justified by the following Lemma, is to bring \eqref{eq:SIR1red3} to a standard form, in order to then apply the entry-exit formula. 

\begin{lemm}\label{lemm:l} Consider \eqref{eq:SIR1red3} and an initial condition $(S_0,I_0)$ with $0<S_0\leq\frac{\gamma}{\beta}-\Delta<S_E$ and $I_0>0$, where $\Delta\in\cO(1)$  and $I_0\in\mathcal{O}(\epsilon)$. Let $0 < \Delta_1 < \Delta$, $\Delta_1\in\cO(1)$, and $(S^*,I^*) $ denote the point where the corresponding trajectory intersects the line $\ell =\left\{ (S,I)\in\mathbb R^2\,|\, S=\frac{\gamma}{\beta} - \Delta_1 \right\}$. Then, for sufficiently small $\epsilon>0$ we have that $I^*$ is exponentially small. Furthermore, the first point at which the trajectory intersects the line $\mathcal{\pi}_\epsilon=\left\{ (S,I)\in\mathbb R^2\,|\, I = I_0 \e^k\right\}$ satisfies $S = S_0 + O(\e \log(\e))$ for $\e \to 0$.
\end{lemm}

\begin{proof}
	We first note that the assumption on $S_0$ simply means that $S_0$ is bounded away from $S_E$ uniformly in $\epsilon$. For the proof it is convenient to define new coordinates $(S,v)$ by $(S,I/\epsilon)=(S,v)$. Then \eqref{eq:SIR1red3} becomes 
\begin{equation}\label{eq:SIR1red5}
\begin{aligned}
S'&=\epsilon(\xi(1- S) -\beta Sv),\\
v'&=v(\beta S-\gamma-\epsilon\xi).
\end{aligned}
\end{equation}
A trajectory of \eqref{eq:SIR1red5} with initial condition $(S_0,v_0)$ with $S_0<\frac{\gamma}{\beta}$ and $v_0=I_0/\epsilon\in\mathcal O(1)$ quickly converges towards and stays $\mathcal{O}(\epsilon)$-close to the $S$-axis for some time. We know from the reduced system that $S'>0$ on the critical manifold, this guarantees that the trajectory crosses the line $\ell$ in a small neighbourhood of the critical manifold. Let $T$ denote the (slow) time it takes the trajectory to reach $\ell$. During such time, $\beta S-\gamma\leq - \beta \Delta_1 < 0$ and therefore 
$$
v'\leq -Kv \implies v\bigg(\frac{T}{\epsilon}\bigg)\leq v_0\e^{-K\frac{T}{\epsilon}} \implies  I\bigg(\frac{T}{\epsilon}\bigg)\leq \epsilon v_0\e^{-K\frac{T}{\epsilon}},
$$
with $K = \beta \Delta_1 >0$. \\
The last claim follows immediately from $v(t)\leq v_0\e^{-Kt}$.
\end{proof}
Note in particular from Lemma \ref{lemm:l} that, before the trajectory intersects $\ell$, its corresponding $I$-coordinate is eventually $\mathcal{O}(\epsilon^2)$, which is what we need for the forthcoming arguments.

\subsection{Applying the entry-exit function}
We are now going to apply the entry-exit formula to describe the way trajectories pass near the non-hyperbolic point $(S,I)=(1/R_0,0)$. 

From Lemma \ref{lemm:Sinf} and \ref{lemm:l}, we can consider an initial point for system \eqref{eq:SIR1red3} with $S_0 <  1/R_0$ and $I_0 = \cO(\e^2)$. Next, we apply a change of variables defined by

\begin{equation}
	S=\frac{u+1}{R_0}, \quad I=\epsilon v,
\end{equation}
which brings the system to a standard form, with $u$ slow and $v$ fast, that is
\begin{equation}\label{eq:uv}
\begin{aligned}
v'&=\gamma(u-\epsilon\xi)v,\\
u'&=\epsilon(\xi(R_0-u-1)-\beta v(u+1)).
\end{aligned}
\end{equation}
So, using the notation of Section \ref{entryexit},
\begin{equation}
	\begin{split}
		f(v,u,\epsilon)&=\gamma(u-\epsilon\xi),\\
		g(v,u,\epsilon)&= \xi(R_0-u-1)-\beta v(u+1),
	\end{split}
\end{equation}
which satisfy the hypotheses of the entry-exit function. Indeed, $S<1$ implies $u<R_0-1$, which means $g(0,u,0)>0$ in the relevant region. Moreover, $f(0,u,0)=\gamma u$, which clearly has the same sign as $u$.

Since $v_0 = I_0/\e = \cO(\e)$, we can now apply the entry-exit formula, which gives $p_0(u_0)$ as the only positive solution of 
\begin{equation}\label{eq:pzero1}
\int_{u_0}^{p_0(u_0)} \frac{u}{R_0-1-u}\textnormal{d}u = 0.
\end{equation}

The integral \eqref{eq:pzero1} can be solved explicitly, giving $p_0(u_0)$ as the positive solution of
\begin{equation}\label{eq:pzero1exp1}
-p_0(u_0)+u_0 - (R_0-1)\ln\bigg(\frac{R_0-1-p_0(u_0)}{R_0-1-u_0}\bigg)=0.
\end{equation}

We now change back to the original $(S,I)$ variables, and introduce, beyond $\Pi_1$ defined in \eqref{Pi1}, the map
\begin{equation}
\label{Pi2}
\Pi_2:\{S\in (0,1/R_0)\} \rightarrow \{S\in (1/R_0,1)\}
\end{equation}
defined by $\dfrac{p_0(u_0)+1}{R_0 }$, where $u_0 = R_0 S_0 - 1$. Combining together the previous results, we can state the following:
\begin{prop}
\label{prop:fast_slow_flow}
Consider the solution of \eqref{eq:SIR1red1} with an initial condition $S_0 > 1/R_0$ and $I_0 = \cO(\e^2)$. Then the orbit $\{S_\e(t), I_\e(t),\ t \in [0,T]\}$ converges for $\e \to 0$ to the union of the orbit under the fast flow 
$$\{(S,I) : \Gamma(S,I) = \Gamma(S_0,0),\ \Pi_1(S_0) \le S \le S_0 \} $$
and under the slow flow
$$ \{(S,0): \ \Pi_1(S_0) \le S \le \Pi_2(\Pi_1(S_0)) \} $$
where $T$ is such that the solution of $S' = \xi(1-S),\ S(0) = \Pi_1(S_0)$ satisfies $S(T) = \Pi_2(\Pi_1(S_0))$.
\end{prop}
The limit orbit is sketched in Figure~\ref{fig:pi1pi2}. Considering the composition of $\Pi_1$ and $\Pi_2$ gives the Poincar\'e map
$$ \Pi: \{S\in [S_E,1), I=I_0\} \to \{\Pi_2(\Pi_1(S))\in [S_E,1), I=I_0\}. $$

\begin{figure}[htbp]\centering
	\begin{tikzpicture}
	\node at (0,0){\includegraphics[scale=3]{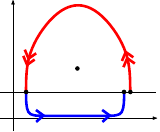}};
	\node at (2.8,-0.8){$I=I_0$};
	\node at (1.2,-0.55){$S_1$}; 
	\node at (-1.35,-0.55){$P_0$};
	\node at (1.8,-1.05){$S_0$};
	\node at (2.6,-1.6){$S$};
	\node at (-2.3,2){$I$};
	\node[text=red] at (0,1.6){$\Pi_1$};
	\node[text=blue] at (0,-1.3){$\Pi_2$};
	\node at (0,-0.4){$(S_E,I_E)$};
	\end{tikzpicture}
	\caption{Sketch of the fast and slow dynamics defining the maps $\Pi_1$ and $\Pi_2$. The fact that $S_1<S_0$ is shown below.}
	\label{fig:pi1pi2}
\end{figure}
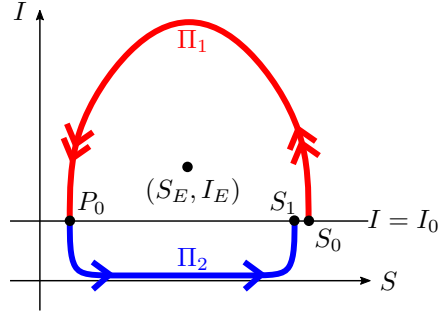

In this notation, we define $P_0=\Pi_1(S_0)$, $S_1=\Pi_2(P_0)=\Pi(S_0)$. These correspond, in the $u$-coordinate, to
$$
u_0=R_0P_0-1\approx R_0 S_\infty -1, \quad p_0(u_0)=R_0 S_1 - 1.
$$

We rewrite~(\ref{eq:pzero1exp1}) as
$$
P_0-S_1 - \bigg(1-\frac{1}{R_0}\bigg) \ln\bigg(\frac{1-S_1}{1-P_0}\bigg)=0.
$$
Which means that $S_1$, the exit point, is the only root greater than $P_0$ of
\begin{equation}\label{eq:effe}
F(x)=x -P_0+ \bigg(1-\frac{1}{R_0}\bigg) \ln\bigg(\frac{1-x}{1-P_0}\bigg).
\end{equation}

It is clear that when the trajectory is in a neighbourhood of $(S_1, I_0)$, as implied by the entry-exit map, one can reapply Proposition \ref{prop:fast_slow_flow}, obtaining $P_1 = \Pi_1(S_1)$ (reached through the fast flow), $S_2 = \Pi_2(P_1)$ (slow flow), and so on, obtaining two sequences 
\begin{equation}
\label{sequences}
 S_0,\ S_1 = \Pi_2(P_0), \ldots, S_n = \Pi_2(P_{n-1}),\ldots \qquad P_0=\Pi_1(S_0)), \ldots, P_n = \Pi_1(S_{n}),\ldots
\end{equation}
\begin{lemm}
\label{lemm:monotone}
The sequence $\{S_n\}$ is decreasing and bounded below by $1/R_0$; the sequence $\{P_n\}$ is increasing and bounded above by $1/R_0$.
\end{lemm}
\begin{proof}
We recall $S_1=\Pi_2(P_0)=\Pi(S_0)$, so if, for any $S_0 \in (1/R_0,1)$, such value is smaller/greater than $S_0$, $\{S_n\}$ is decreasing/increasing.\\
We notice that $\Pi(S_0)<S_0$ if and only if $\Pi_2(P_0)<\Pi_1^{-1}(P_0)$, where $\Pi_1^{-1}(P_0)>P_0$ is the only such root of
\begin{equation}\label{eq:gggi}
G(x)=x-P_0+\frac{1}{R_0}\ln\bigg(\frac{P_0}{x}\bigg),
\end{equation}
which comes from $\Gamma(x,0)=\Gamma(P_0,0)$; we recall that $\Gamma$ describes the trajectories of the layer equation. The functions $F$ and $G$ are sketched in Figure~\ref{fig:effgi}. 

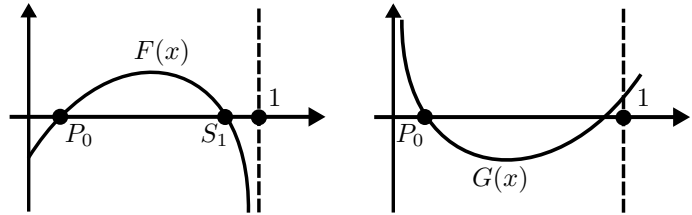
\begin{figure}[htbp]\centering
	\begin{tikzpicture}
		\node at (0,0){\includegraphics[scale=2]{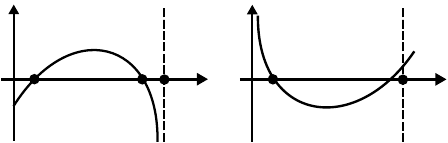}};
		\node at (3.9,0.1){$1$};
		\node at (-2.5,0.7){$F(x)$};
		\node at (-1,0.1){$1$};
		\node at (-1.8,-0.4){$S_1$};
		\node at (-3.6,-0.4){$P_0$};
		\node at (0.8,-0.4){$P_0$};
		\node at (2,-1){$G(x)$};
	\end{tikzpicture}
	\caption{Sketch of the functions $F$ and $G$, which implicitly define $\Pi_2$ and $\Pi_1^{-1}$, respectively.}
	\label{fig:effgi}
\end{figure}

Then, since $G$ is increasing for $x>1/R_0$,
$$
\Pi_2(P_0)<\Pi_1^{-1}(P_0) \iff G(\Pi_2(P_0))<0 .
$$

The fact that $S_1<S_0$ can be shown as a particular case of the following, more general proposition, by taking $a=P_0$, $b=1/R_0$, $x^*=S_1$.
\begin{lemm}
	Let $0<a<b<1$, $F(x)=x-a+(1-b)\ln(\frac{1-x}{1-a})$, $G(x)=x-a+b\ln(\frac{a}{x})$. Let $x^* \in (a,1)$ be the only zero greater than $a$ of $F$. Then $G(x^*)<0$.
\end{lemm}
\begin{proof} 
We use the auxiliary function $H(x)=F(x)+\frac{b}{1-b}G(x)$, which, under the hypotheses, is decreasing for $x\in(0,1)$. Next we have that $H(a)=F(a)+\frac{b}{1-b}G(a)=0$ which implies 
$$
0>H(x^*)=F(x^*)+\frac{b}{1-b}G(x^*)=\frac{b}{1-b}G(x^*) \implies G(x^*)<0. 
$$
\end{proof}

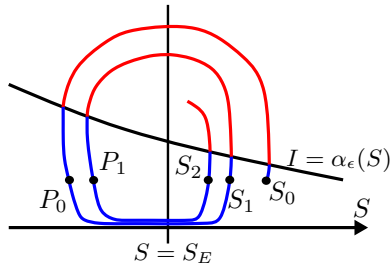
\begin{figure}[htbp]\centering
	\begin{tikzpicture}
	\node at (0,0){\includegraphics[scale=1.6]{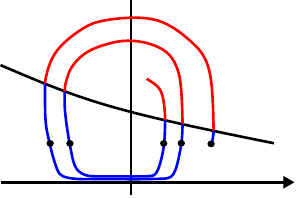}};
	\node at (1.25,-.9){$S_0$};
	\node at (0.7,-1){$S_1$};
	\node at (0,-0.55){$S_2$};
	\node at (-1,-0.5){$P_1$};
	\node at (-1.8,-1){$P_0$};
	\node at (2,-0.45){\small $I=\alpha_\epsilon(S)$};
	\node at (2.3,-1.1){$S$};
	\node at (-0.2,-1.7){\small $S=S_E$};
	\end{tikzpicture}
		\caption{$\alpha_\epsilon(S)=\mathcal{O}(\epsilon)$; the red parts of the orbit are fast for both variables, the blue parts are fast for $I$, slow for $S$.}
		\label{fig:duegiri}
\end{figure}
Since $\Pi_1$ is a decreasing function, from the fact that $\{S_n\}$ is decreasing, it follows that $\{P_n\}$ is increasing.
\end{proof}
\begin{prop}
The sequences $\{S_n\}$ and $\{P_n\}$ defined in \eqref{sequences} both converge to $1/R_0$.
\end{prop}
\begin{proof}
The convergence can be shown reasoning by contradiction, for example by looking at the sequence $S_i$. We know it is decreasing, and bounded below by $1/R_0$, so if it is not converging to this value, it must be converging to some other value $S_{\lim}>1/R_0$.
But if this is the case, $\Pi(S_{\lim})<S_{\lim}$, which contradicts the nature of $S_{\lim}$.\\
 Completely analogously we can see that $P_i \rightarrow 1/R_0$.
 \end{proof}
Extending Proposition \ref{prop:fast_slow_flow}, one can easily show that, if $S_0 > 1/R_0$ and $I_0 = \cO(\e^2)$, the orbits $\{S_\e(t), I_\e(t), t \in [0,T]\}$ for any $T$ converge for $\e \to 0$ to a finite union of orbits (under the fast flow) from $(S_n,0)$ to $(P_n,0)$, and slow flows on the $S$-axis from $(P_n,0)$ to $(S_{n+1},0)$.

The same can be shown for any initial condition, since starting from any $(S_0,I_0)$ with $I_0 > 0$, the solutions will approach a point $(S_\infty,0)$ with $S_\infty < 1/R_0$, so that setting $P_0 = S_\infty$, one can repeat the above argument.

What can we say of the orbits $\{S_\e(t), I_\e(t)\}$ for $\e$ small but fixed as $t \to \infty$? When $1/R_0 - P_n = \cO(\e)$, the argument of Lemma \ref{lemm:l} does not work. Hence, we cannot say, and indeed it is no longer true, that $I(t)$ becomes $\cO(\e^2)$ afterwards, and we cannot apply the entry-exit Lemma as above.

However, the previous argument shows that $\{S_\e(t), I_\e(t)\}$ reaches an $\e$-neighbourhood of the equilibrium $(S_E,I_E)$. Linearization at the equilibrium then shows that all trajectories of \eqref{eq:SIR1red3} starting in the set $\{(S,I)\in \mathbb{R}^2 \, |\, S\geq 0, I>0, S+I \leq 1 \}$ converge towards $(S_E,I_E)$, as already known (Theorem \ref{theor:convergence}). This analysis provides an alternative proof, valid for $\e > 0$ sufficiently small.

Biologically, the above analysis tells us that between two consecutive peaks of infection there is a long ($\mathcal{O}(1/\epsilon)$) time during which the fraction of infected population is exponentially small. On the other hand, the duration of high infected portion of the population is rather small (it occurs on the fast time scale). Ultimately, however, under the setting of this section the only possible asymptotic outcome is convergence towards the endemic equilibrium $(S_E,I_E)$ via damped oscillations. 

\subsection{SIRS model}  \label{SIRS}
We now consider a SIRS compartment model. The SIRS model is a slight modification of the SIR model and thus we keep the same notation.  The SIRS model is given by the following system:\\
\begin{minipage}{.4\textwidth}			

		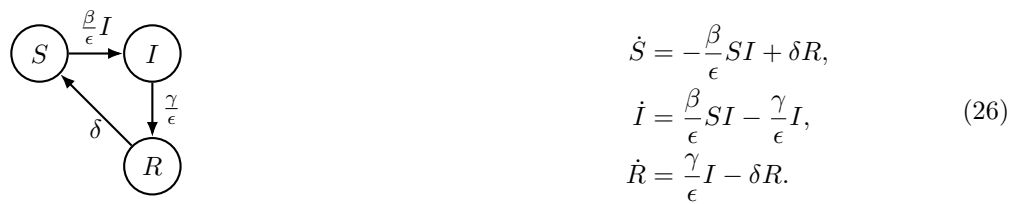
\begin{figure}[H]\centering
\begin{tikzpicture}
		\node[draw,circle,thick,minimum size=.75cm] (s) at (1,0) {$S$};
		\node[draw,circle,thick,minimum size=.75cm] (i) at (2.5,0) {$I$};
		\node[draw,circle,thick,minimum size=.75cm] (r) at (2.5,-1.5) {$R$};
		\draw[-{Latex[length=2.mm, width=1.5mm]},thick] (s)--(i) node[above, midway]{$\frac{\beta}{\epsilon}I$};
		\draw[-{Latex[length=2.mm, width=1.5mm]},thick] (i)--(r) node[right, midway]{$\frac{\gamma}{\epsilon}$};
		\draw[-{Latex[length=2.mm, width=1.5mm]},thick] (r)--(s) node[below, midway]{$\delta$};
		\end{tikzpicture}
			\caption{Flow diagram for~(\ref{eq:SIR2}).}
		\end{figure}	
	\end{minipage}
		\hfill
\begin{minipage}{.5\textwidth}
\begin{equation}\label{eq:SIR2}
\begin{aligned}
\dot{S}&= -\frac{\beta}{\epsilon}SI+\delta R,\\
\dot{I}&=\frac{\beta}{\epsilon}SI-\frac{\gamma}{\epsilon}I,\\
\dot{R}&=\frac{\gamma}{\epsilon}I-\delta R.
\end{aligned}
\end{equation}
\end{minipage}\\
\\

In this model there is no birth nor death, so the population remains constant. The small positive parameter $0< \epsilon \ll 1$ gives rise to the difference in magnitude between the large infection rate $\beta /\epsilon$, the large recovery rate $\gamma/\epsilon$ and the rate of loss of immunity $\delta$. This difference models a highly contagious disease with a short infection period with possibility of reinfection. The main distinctions with the SIR system presented in Section \ref{SIR} are the absence of demographic dynamics (no birth/death) and the possible loss of immunity (meaning that individuals can move from $R$ to $S$). As we will see shortly, however, this important biological difference does not modify the qualitative behaviour of the system.

As we noticed in Section \ref{SIR}, $\dot{N}=\dot{S}+\dot{I}+\dot{R}=0$, that is, the total population remains constant, so we assume without loss of generality $N(0)=1$, which implies $N(\tau)\equiv 1$ for all $\tau \geq 0$; this allows us, using $R=1-S-I$, to reduce the system to
\begin{equation}\label{eq:SIR2red1}
\begin{aligned}
\dot{S}&= -\frac{\beta}{\epsilon}SI+\delta (1-S-I),\\
\dot{I}&=\frac{\beta}{\epsilon}SI-\frac{\gamma}{\epsilon}I.
\end{aligned}
\end{equation}
Proceeding as in the first model, we introduce the fast time variable $t=\tau/\epsilon$, which gives
\begin{equation}\label{eq:SIR2red3}
\begin{aligned}
S'&=-\beta SI +\epsilon\delta (1-S-I),\\
I'&=I(\beta S-\gamma),
\end{aligned}
\end{equation}
where now the prime ( $'$ ) indicates the derivative with respect to $t$.

The critical manifold is, as before, the set $\mathcal{C}_0=\{(S,I)\in \mathbb{R}^2 \,|\, I=0\}$, and the slow flow along it is given by $\dot{S}=\delta(1-S)$, which implies flow towards the point $(S,I)=(1,0)$.

The $\epsilon \rightarrow 0$ limit system corresponding to \eqref{eq:SIR2red3} is
\begin{equation}\label{eq:SIR2redlim}
\begin{aligned}
S'&=-\beta SI,\\
I'&=I(\beta S-\gamma),
\end{aligned}
\end{equation}
which is exactly the limit system we obtained in Section \ref{SIR}. Hence, we can apply the same qualitative reasoning as before, with some small changes: in the perturbed system the nullcline for $S$ is slightly different, giving $I=\alpha(S)=(\epsilon\delta(1-S))/(\beta S+\epsilon\delta)$, and the value of $S_E$ is exactly $1/R_0$.

The previous ansatz for the Lyapunov function does not work here; we could find another one, following what was done in \cite{Lyap}, but we instead follow the analysis with the entry-exit function which, as we show below, does not change.

The trajectory starting from $(S_0,I_0)$, with $I_0\in\mathcal{O}(\epsilon^2)$, follows the same qualitative behaviour: after it intersects $I=\alpha(S)$ at a point $(S_\infty+\mathcal{O}(\epsilon),\mathcal{O}(\epsilon))$, it eventually intersects the horizontal line $I=I_0$. At that moment, we change the variables as before:
$$
S=\frac{u+1}{R_0}, \quad I=\epsilon v,
$$
and we obtain a system in standard form:
\begin{equation}\label{eq:uv2}
\begin{aligned}
v'&=\gamma uv,\\
u'&=\epsilon(-\beta v(u+1)+\xi(R_0-u-1-\epsilon v)).
\end{aligned}
\end{equation}
In the notation of the entry-exit function, then,
\begin{equation}
\begin{aligned}
f(v,u,\epsilon)&=\gamma u,\\
g(v,u,\epsilon)&=-\beta v(u+1)+\xi(R_0-u-1-\epsilon v),
\end{aligned}
\end{equation}
which satisfy the hypotheses in the relevant region; hence, we can compute $p_0(u_0)$ with exactly the same integral equation
\begin{equation}\label{eq:pzero2}
\int_{u_0}^{p_0(u_0)} \frac{u}{R_0-1-u}\textnormal{d}u = 0,
\end{equation}
and the procedure we followed for the SIR model can be applied to this SIRS one identically to show the global convergence to the unique equilibrium.

By following a similar analysis as the one performed so far one can also show that considering a SIRS model with demography would not change the qualitative behaviour of the system. 

The results obtained so far for the SIR and SIRS models are summarized in the following Proposition.
\begin{prop} The SIR, SIRS without and with demographic dynamics, with infection and recovery rates $\mathcal{O}(1/\epsilon)$ big compared to the other parameters, are all qualitatively equivalent. Their main common features are: 
	\begin{itemize}[leftmargin=*]
		\item boundedness of solutions in the set $\left\{ (S,I,R)\in\mathbb R^3_{\geq0}\, | \, \, 0\leq S+I+R\leq 1 \right\}$,
		\item population either constant, or converging uniformly and exponentially fast to a constant, which allows to reduce the number of compartments from $3$ $(S,I,R)$ to $2$ $(S,I)$,
		\item existence of an endemic equilibrium point of the form $(S_E,I_E)=(\frac{1}{R_0}+\mathcal{O}(\epsilon), \mathcal{O}(\epsilon))$,
		\item fast-slow decomposition in the $I$ and $S$ coordinate, respectively, $\mathcal{O}(\epsilon)$-close to the critical manifold $\cC_0=\left\{ (S,I)\in[0,1]^2\,|\,I=0 \right\}$,
		\item counterclockwise spiralling of the orbits towards $(S_E,I_E)$, and consequent absence of periodic orbits.
	\end{itemize}
These common features mean that, in the long run, the population in each of these models converges to an equilibrium $\mathcal{O}(\epsilon)$ close to $(S,I,R)=(1/R_0,0, 1-1/R_0)$, in the first octant of $\mathbb{R}^3$; each of the three variables have damped oscillations around the equilibrium value.
\end{prop}

In the next section we study a more complete (but also more complicated) epidemic model, where the techniques developed so far shall be extended. 

\subsection{SIRWS model}\label{sec:SIRWS}
We consider the SIRWS compartment model suggested by Dafilis et al.~in~\cite{Dafi}. As in the previous models, we assume that some parameters are $\mathcal{O}(\epsilon)$ small compared to others, making the corresponding processes slow, and the remaining ones fast (the changes correspond to every occurrence of $\epsilon$ in system~(\ref{eq:SIRWS})). This allows us to build on the analysis done in sections \ref{SIR} and \ref{SIRS}, and to apply the entry-exit function to a more challenging model.

The model we are concerned with in this section is given by:\\
\begin{minipage}{.4\textwidth}\centering
	\begin{figure}[H]
		\centering
		\begin{tikzpicture}
		\node[draw,circle,thick,minimum size=.75cm] (s) at (1,0) {$S$};
		\node[draw,circle,thick,minimum size=.75cm] (i) at (2.5,0) {$I$};
		\node[draw,circle,thick,minimum size=.75cm] (r) at (2.5,-1.5) {$R$};
		\node[draw,circle,thick,minimum size=.75cm] (w) at (1,-1.5) {$W$};
		\draw[-{Latex[length=2.mm, width=1.5mm]},thick] (s)--(i) node[above, midway]{$\frac{\beta}{\epsilon}I$};
		\draw[-{Latex[length=2.mm, width=1.5mm]},thick] (i)--(r) node[right, midway]{$\frac{\gamma}{\epsilon}$};
		\draw[{Latex[length=2.mm, width=1.5mm]}-,thick] (s)--++(-1,0) node[above,midway]{$\xi$};
		\draw[-{Latex[length=2.mm, width=1.5mm]},thick] (s)--++(0,1) node[left,midway]{$\xi$};
		\draw[-{Latex[length=2.mm, width=1.5mm]},thick] (i)--++(1,0) node[above,midway]{$\xi$};
		\draw[-{Latex[length=2.mm, width=1.5mm]},thick] (r)--++(1,0) node[above,midway]{$\xi$};
		\draw[-{Latex[length=2.mm, width=1.5mm]},thick] (w)--(s) node[left,midway]{$2\kappa$};
		\draw[-{Latex[length=2.mm, width=1.5mm]},thick] (w)--++(-1,0) node[above,midway]{$\xi$};
		\draw[{Latex[length=2.mm, width=1.5mm]}-,thick] ($(w)+(25:.75/2)$)--($(r)+(155:.75/2)$) node[above,midway]{$2\kappa$};
		\draw[-{Latex[length=2.mm, width=1.5mm]},thick] ($(w)+(-25:.75/2)$)--($(r)+(-155:.75/2)$) node[below,midway]{$\nu\frac{\beta}{\epsilon}I$};

		\end{tikzpicture}

			\caption{Flow diagram for~(\ref{eq:SIRWS})}
		\end{figure}
\end{minipage}
\hfill
\begin{minipage}{.5\textwidth}
\begin{equation}\label{eq:SIRWS}
\begin{aligned}
\dot{S}&= -\frac{\beta}{\epsilon}SI+2\kappa W +\xi (1-S),\\
\dot{I}&=\frac{\beta}{\epsilon}SI-\frac{\gamma}{\epsilon}I-\xi I,\\
\dot{R}&=\frac{\gamma}{\epsilon}I-2\kappa R +\nu\frac{\beta}{\epsilon}IW- \xi R,\\
\dot{W}&=2\kappa R-2\kappa W-\nu\frac{\beta}{\epsilon}IW-\xi W.
\end{aligned}
\end{equation}
\end{minipage}

\vspace{.5cm}
As in the previous models, susceptible individuals ($S(\tau)$) become infectives ($I(\tau)$) upon contact with infectious individuals, who, at rate $\gamma/\epsilon$ become immune at their first stage ($R(\tau)$), and then, at a rate $2\kappa$, become second-stage (`weakly') immune ($W(\tau)$). Weakly immune individuals may then lose totally their immunity at rate $2\kappa$, or, upon contact with infectious individuals, revert back to fully immune individuals ($R(\tau)$), thanks to the so-called immunity boosting. The constant $\nu$ is the ratio between the rate at which immunity boosting occurs in weakly immune individuals, and the rate at which susceptibles become infected. Finally, we assume a constant birth rate $\xi$, equal to the death rate, and that all individuals are born susceptible. Through the introduction of the small parameter $\epsilon$ we consider a highly contagious disease with a very short infection period, compared to other typical times of the system; indeed, the average length of the infectious period is  $\epsilon/\gamma$, while the average length of life is $1/\xi$ and the total average length of the immune period is $1/\kappa$ for individuals whose immunity is not boosted. Such relation between the parameters has been assumed, for example, for diseases such as pertussis, as described in \cite{Lavi}, where the authors estimated $\beta =260$, $\gamma=17$, $\xi = 0.01$, $\kappa =0.1$, $\nu=20$; hence, the analysis which follows may be useful in the modelling of such diseases.

Analogous to the previous models, the set $\left\{ (S,I,R,W)\in\mathbb R^4_{\geq0}\, | \, \, 0\leq S+I+R+W\leq 1 \right\}$ is invariant. We can thus scale the total population to 1, so that we can use $R=1-S-I-W$. We notice that system~(\ref{eq:SIR1}) can be recovered from system~(\ref{eq:SIRWS}) by setting $\kappa=\nu=0$, and ignoring the consequently decoupled $W$ coordinate.

As we shall describe in our analysis below, incorporating the waning state $W$ modifies considerably the dynamics of the model; in fact, it induces the possibility of periodic limit cycles, a feature that the previous simpler models did not have. This is particularly important when comparing the dynamics of the SIRWS model with that of the SIRS model where, even if recovered portions of the population may become again susceptible, there is still no ``long run periodic behaviour''.

As we have done before, introducing the fast time variable $t=\tau/\epsilon$ brings the system into the form
\begin{equation}\label{eq:SIRWSslow}
\begin{aligned}
S'&= -\beta SI+\epsilon (2\kappa W +\xi (1-S)),\\
I'&=\beta SI-\gamma I-\epsilon  \xi I,\\
R'&=\gamma I+\nu\beta IW-\epsilon (2\kappa R+ \xi R), \\
W'&=-\nu\beta IW +\epsilon (2\kappa R-2\kappa W-\xi W).
\end{aligned}
\end{equation}

\begin{remark}
	Note that the critical manifold is (similarly to the previous models) given by
	\begin{equation}
		\cC_0=\left\{ (S,I,R,W)\in[0,1]^4\,|\, I=0 \right\}.
	\end{equation}
\end{remark}

Furthermore, in the $\epsilon\rightarrow 0$ limit, $S$ and $I$ become independent of $R$ and $W$, and orbits follow the same behaviour we have seen in the fast phases of the first two models. In other words, the $(S,I)$-orbits of the layer equation follow a power level of $\Gamma(S,I)=\gamma \ln(S)-\beta (S+I)$, and converge towards $(S_\infty,0)$\footnote{We recall that $S_\infty$ is defined as the nontrivial solution of $\Gamma(S,0)=\Gamma(S_0,0)$.}. These observations motivate the following lemma.

\begin{lemm}\label{Lemma5} Consider the layer equation corresponding to \eqref{eq:SIRWSslow}. Then, as $(S,I)\to(S_\infty,0)$ one has $W\to W_\infty:=W_0 \exp^{-\nu R_0(S_0+I_0-S_\infty)}$, where $W_0=W(0)$.
	
\end{lemm}

\begin{proof}
We note that 
$$
\int_{0}^{\infty} \bigg(S'(u)+I'(u)\bigg)\textnormal{d}u = -\gamma \int_{0}^{\infty} I(u) \textnormal{d}u \implies S_0+I_0-S_\infty = \gamma \int_{0}^{\infty} I(u) \textnormal{d}u,
$$
due to the fact that $\lim_{t\rightarrow +\infty} I(t)=0$. Next, note from \eqref{eq:SIRWSslow} that in the limit $\epsilon=0$ one has $\frac{W'}{W}=-\nu\beta I$, which implies $W(t)=W_0\exp^{-\nu \beta\int_{0}^{t} I(u) \textnormal{d}u}$. Letting $t\to \infty$ leads to the result, recalling that $R_0=\frac{\beta}{\gamma}$.
\end{proof}

Since we have already shown that the layer equation is in the $(S,I)$-coordinates the same as before, we proceed just in the same way, that is, we apply first the change of coordinates

$$
S=\frac{u+1}{R_0}, \quad I=\epsilon v,
$$
which gives a system in standard singular perturbation form, with $u, W$ slow and $v$ fast, namely
\begin{equation}\label{eq:uvW1}
\begin{aligned}
v'&=(\gamma u-\epsilon \xi)v=: f(v,u,\epsilon)v,\\
u'&=\epsilon(-\beta v(u+1)+2\kappa R_0 W+\xi(R_0-u-1))=:\epsilon g(v,u,W,\epsilon),\\
W'&=\epsilon(-\nu\beta v W+ 2\kappa -2\kappa \frac{u+1}{R_0} -4\kappa W -\xi W)+\mathcal{O}(\epsilon^2).
\end{aligned}
\end{equation}

And, accordingly, in the slow time scale $\tau$:
\begin{equation}\label{eq:uvW2}
\begin{aligned}
\epsilon\dot{v}&=(\gamma u-\epsilon \xi)v,\\ 
\dot{u}&=-\beta v(u+1)+2\kappa R_0 W+\xi(R_0-u-1),\\
\dot{W}&=-\nu\beta v W+ 2\kappa -2\kappa \frac{u+1}{R_0} -4\kappa W -\xi W+\mathcal{O}(\epsilon).
\end{aligned}
\end{equation}

Naturally, the critical manifold in these new coordinates is $\mathcal{C}_0=\left\{(u,v,W)\in \mathbb{R}^3 \, | \, v=0\right\}$. 

In order to use the entry-exit formula, as described in \cite[ equation (12)]{HsuRuan}, we first check that indeed
\begin{equation}\label{eq:hypo}
\begin{aligned}
g(0,u,W,0)&=2\kappa W R_0 +\xi(R_0-u-1)>0, \\
f(0,u,0)&=\gamma u \lessgtr 0 \iff u \lessgtr 0.
\end{aligned}
\end{equation}

However, the presence of $W$ in the equation for $\dot{u}$ makes the entry-exit integral 
\begin{equation}\label{eq:pzero3}
\int_{u_0}^{p_0(u_0)} \frac{u}{2\kappa W(u) R_0+\xi(R_0-u-1)}\textnormal{d}u = 0
\end{equation}
not immediately computable, as we would need to find and expression for $W(u)$. To deal with this issue, let us look at the $(S,W)$-dynamics in the slow time variable $t$ on the critical manifold $I=0$:
\begin{equation}\label{eq:SW}
\begin{aligned}
\dot{S}&=2\kappa W +\xi(1-S),\\
\dot{W}&=2\kappa (1-S)-(4\kappa +\xi)W.\\
\end{aligned}
\end{equation}
This system of ODEs can be solved explicitly, assuming initial conditions \linebreak $(S(0),W(0))=(S_\infty,W_\infty)$, the limit values of the fast loop, we have:
\begin{equation}\label{eq:SWsol}
\begin{aligned}
S(\tau)&=1+[S_\infty -1+2\kappa(S_\infty+W_\infty-1)\tau]\exp(-(2\kappa+\xi)\tau),\\
W(\tau)&=[W_\infty-2\kappa(S_\infty+W_\infty-1)\tau]\exp(-(2\kappa+\xi)\tau)\\
&=1-S(\tau)-(1-S_\infty-W_\infty)\exp(-(2\kappa+\xi)\tau).
\end{aligned}
\end{equation}
The phase-portrait of \eqref{eq:SW} is illustrated in Figure~\ref{fig:sloflo}, where the only feasible region is the triangle $0\leq S+W \leq 1$, $S,W\geq 0$, and all trajectories converge to $(S,W)=(1,0)$.
\begin{figure}[H]\centering
	\begin{tikzpicture}
		\node at (0,0){
			\includegraphics[scale=1]{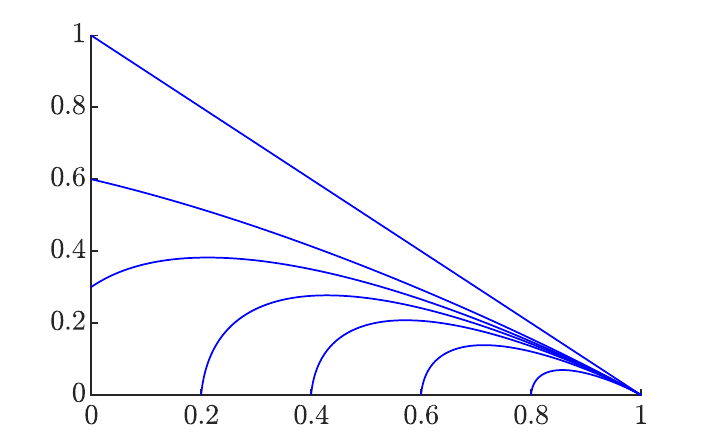}
		};
		\node at (0,-2.5) {$S$};
		\node at (-3.5,0) {$W$};
	\end{tikzpicture}
	\caption{Phase plane for the $S,W$ couple; values for $\kappa=0.1$ and $\xi=0.0125$ taken from \cite{Dafi}}	
		\label{fig:sloflo}
\end{figure}


Note that, in general, the integral \eqref{eq:pzero3} is not explictly computable. Hence,  let $\textnormal{d}u=[2\kappa R_0 W+\xi(R_0-u-1)]\textnormal{d}\tau$; then one can transform \eqref{eq:pzero3} into an integral equation which provides the exit time $T_E$, namely, after substituting $\textnormal{d}u=[2\kappa R_0 W+\xi(R_0-u-1)]\textnormal{d}\tau$ in \eqref{eq:pzero3} one has
$$
\begin{aligned}
&\int_{0}^{T_E} u(\tau)\textnormal{d}\tau= 0.
\end{aligned}
$$

In other words, $T_E$ is defined as the time it takes to go from $u=u_0$ to $u=p_0(u_0)$, and therefore it is also the time during which a trajectory of \eqref{eq:SIRWSslow} stays $\mathcal{O}(\epsilon^2)$-close to the critical manifold. This implies, remembering $u(\tau)=R_0S(\tau)-1$, that
\begin{equation}\label{eq:pzerot}
\int_{0}^{T_E} (R_0S(\tau)-1)\textnormal{d}\tau= 0.
\end{equation}
Using the explicit equation for $S(\tau)$ given in~(\ref{eq:SWsol}), and introducing, for ease of notation, $A:=2\kappa+\xi$, $B:=2\kappa(S_\infty+W_\infty-1)$, 
$C:=S_\infty-1$ so that
$$
S(\tau)=1+C\exp(-A\tau)+Bt\exp(-A\tau),
$$
the equation for the exit time $T_E$~(\ref{eq:pzerot}) becomes
\begin{equation}\label{eq:TEE}
-\frac{R_0\exp(-AT_E)(ABT_E+AC+B)}{A^2}+(R_0-1)T_E+\frac{R_0(AC+B)}{A^2}=0.
\end{equation}
Clearly $T_E=0$ is a solution. Moreover, there is only one strictly positive solution, since $S(\tau)$ is strictly increasing and tends to $1$ as $\tau\rightarrow+\infty$. Such solution provides the exit time.

Substituting the positive solution $T_E$ of~(\ref{eq:TEE}) it in~(\ref{eq:SWsol}) we obtain the exit point $(S(T_E),W(T_E))$. However, due to the implicit formulae we have obtained above, such a computation is only suitable numerically (see Section \ref{sec:po}). Despite the previous obstacle, we can still check how the exit points depend on certain parameters. For example, from the first equation of~(\ref{eq:SWsol}) we observe that
\begin{equation}\label{eq:dsdxi}
	\frac{\partial S}{\partial \xi} (\tau, \xi) = -\tau[S_\infty - 1 +2\kappa(S_\infty + W_\infty -1)\tau]\exp^{-(2\kappa+\xi)\tau}>0,
\end{equation}
which immediately suggests that the exit time is decreasing in $\xi$. Namely,  let $T_{E,i}$ denote the exit time with $\xi=\xi_i$ and $i=1,2$. If $\xi_1<\xi_2$ then, using \eqref{eq:dsdxi}, one sees that $T_{E,1}>T_{E,2}$.

To provide more insight on the dynamics of the SIRWS model, we are now going to complement our previous study with a numerical analysis, where the computed exit time $T_E$ shall play an essential role.

\subsubsection{Periodic orbits}\label{sec:po}

Recall that in the SIR and SIRS models no periodic trajectories are possible. In this section we show that the SIRWS does have periodic solutions, and of particular biological relevance, stable limit cycles. Our motivation is that if a stable limit cycle exists, then a disease would have periodic outbursts. Furthermore, due to the time scales present in the model, there is the danger of missing such periodicity if only short time scale analysis is considered. Moreover, information regarding the parameter regions in which damped/sustained oscillations occur can give directions as to which parameter(s) to modify in order to have a desired control of the epidemic.

As it is usual in GSPT, the general idea to show existence of limit cycles of the perturbed (fast-slow) system is to first find a singular cycle, see for example \cite{KoSz,taghvafard2019geometric}. A singular cycle is a concatenation of limiting slow and fast orbits that form a cycle. Afterwards, given that some conditions are met, we argue that such singular cycle gives rise to a limit cycle of the fast-slow system. We further remark that a mixture of analytical and numerical methods is relevant since we have to combine local analytical results with global numerical results, which is a key theme in multiple time scale systems~\cite{GuckenheimerWechselbergerYoung,Haiduc1,KuehnRetMaps}.

The steps to form a singular cycle of the SIRWS model are as follows:

\begin{enumerate}[leftmargin=*]
	\item Choose a section $J_1=\left\{ (S,I,W)=(S_0,0,W)\, |\, S_0>\frac{1}{R_0}, \, W\in(0,1-S_0) \right\}$. This section is transversal to the reduced slow flow and is located on the unstable region of the critical manifold.
	\item Consider the map $\Pi_1$ defined by the layer equation. Under such a map one obtains a new section on the critical manifold $J_2:=\Pi_1(J_1)$. The coordinates of $J_2$ are given by $(S_\infty,0,W_\infty)$, as in Lemma \ref{Lemma5}.
	\item Consider the map $\Pi_2$ defined by the slow flow \emph{for a time} $T_E$ implicitly given by \eqref{eq:TEE}, i.e. $\Pi_2(J_2) = (S(T_E),W(T_E))$ with $(S(\tau),W(\tau))$ given by \eqref{eq:SWsol}, and let $J_3:=\Pi_2(J_2)$. Recall from the last part of section \ref{sec:SIRWS} that we can tune the exit time, for example, by changing the parameter $\xi$, without changing the map $\Pi_1$.
	\item If $J_3$ intersects transversally $J_1$, then we have a robust singular cycle given precisely by the orbit corresponding to a fixed point of $\Pi_2 \circ \Pi_1$, see Figure \ref{fig:sing_cycle} for a schematic representation of these four arguments. 

In the present context, robust means that the singular cycle persists under small smooth perturbations as a periodic orbit of the fast-slow system precisely due to the transverse intersection of $J_1$ and $J_3$ \cite{thom1954quelques} (if it occurs).

It is clear that for the particular SIRWS model, there is a priori no guarantee that such a transverse intersection occurs for a particular set of parameters and initial conditions. To clarify that indeed such a fixed point exists upon variation of parameter values, we refer to the situation shown in Figure \ref{fig:numerics} varying the parameter $\xi$, we argue as follows: let 
	$ F^\xi = (F_1^\xi,F_2^\xi)= \Pi_2 \circ \Pi_1 :\cC_0\to\cC_0$ using the parameter $\xi$, and $X = \{\xi\ :\ J_3 \cap J_1 \not = \emptyset\}$. We can then define, for $\xi \in X$, $\bar w (\xi)$ as the value of $w$ such that $F_1^\xi(S_0,w) = S_0$. Note moreover that for all $w$, the inequalities $0 < F_2^\xi(S_0,w) < 1 - S_0$ hold, as can be seen by \eqref{eq:SWsol}.\\
	Consider finally 
	$$ g : X \to \mathbb{R},\quad g(\xi) = \bar w(\xi) - F_2^\xi(S_0,w) $$
	If $X = [\xi_1,\xi_2]$, we have $\bar w(\xi_1)=0$ and $\bar w(\xi_2)=1-S_0$, or vice versa. Hence $g(\xi_1) < 0 < g(\xi_2)$, or vice versa. In either case, there exists $\bar \xi \in (\xi_1,\xi_2)$ such that $g(\bar \xi) = 0$, i.e. $F_1^{\bar \xi}(S_0,\bar w(\bar \xi)) = S_0$ and $F_2^{\bar \xi}(S_0,\bar w(\bar \xi)) = \bar w(\bar \xi)$ as claimed.\\

Moreover, since we know that both $\Pi_1$ and $\Pi_2$ are contractions in the $W$-direction (refer to \eqref{eq:SIRWSslow} and to Figure \ref{fig:sloflo}), such a singular cycle is locally attracting. Hence it persists as a locally attracting periodic orbit for $\epsilon>0$ sufficiently small. We remark, however, that this does not mean that there are no other limit cycles for $\epsilon>0$ sufficiently small. As we show in our numerical analysis of the forthcoming section, there is in fact a range of parameter for which a stable and an unstable limit cycle co-exist. The existence of the unstable limit cycle, however, does not follow from our previous perturbation arguments.
\end{enumerate}

\begin{figure}[htbp]\centering
	\begin{tikzpicture}
		\node at (0,0){
		\includegraphics[scale=1.75]{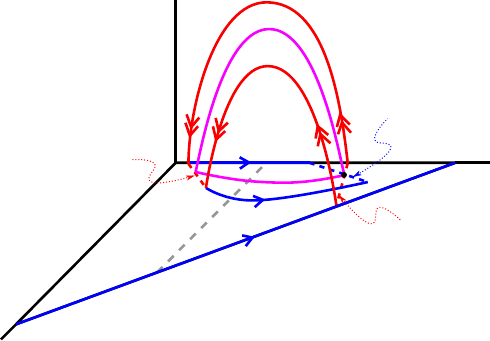}
		};
		\node at (4.5,0.1) {$S$};
		\node at (-1.25,3.25) {$I$};
		\node at (-4.5,-3.25) {$W$};
		\node at (3.,-1) {$J_1$};
		\node at (2.5,1.1) {$J_3$};
		\node at (-2.25,0.2) {$J_2$};
		\node at (-2.9,-2.) {$\cC_0^{\textnormal{a}}$};
		\node at (-.45,-1.1) {$\cC_0^{\textnormal{r}}$};
	\end{tikzpicture}
	\caption{Schematic representation of the singular cycle, shown in magenta. The red arrows depict the map $\Pi_1:(S_0,W_0)\mapsto (S_\infty,W_\infty)$ so that $\Pi_1(J_1)=J_2$. The blue arrows depict the map $\Pi_2$ given by the reduced flow on $\cC_0$ and induced by \eqref{eq:SW} (for a finite time $T_E(S_\infty,W_\infty)$) so that $\Pi_2(J_2)=\Pi_2(\Pi_1(J_1))=J_3$. If the sections $J_1$ and $J_3$ intersect, then such an intersection defines closed singular orbits. If $J_1$ and $J_3$ intersect transversally, then such intersection persists for $\epsilon>0$ sufficiently small giving rise to a periodic orbit of the SIRWS model.}
	\label{fig:sing_cycle}
\end{figure}
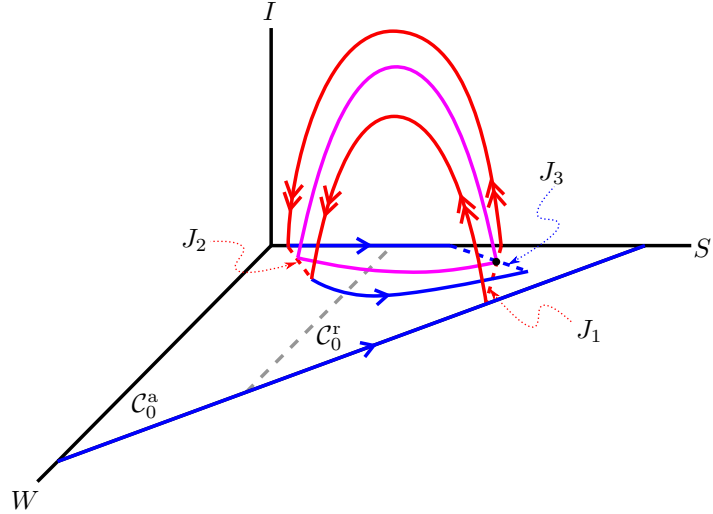

Naturally, the above procedure is only sufficient to show existence of limit cycles that pass close to the critical manifold and provides no information on other possible limit cycles of the fast-slow system, compare with \cite{PMID:29434506}. Yet our attention is precisely focused on describing those limit cycles arising from the time scale separation.

An example of the above procedure is shown in Figure \ref{fig:numerics} where we set $\{\beta =260, \gamma =17, \kappa=0.1, \xi=0.0125, \nu=5 \}$, values taken from \cite{Dafi}. Figures in the left column show the evolution of $J_1$ (dashed red) in the fast system (red) and of $J_2$, too small to be visible, in the slow system (blue). Figures in the right column zoom to the interval $J_3$ (blue) for each parameter value, and its position relative to $J_1$ (dashed red). Note that
\begin{itemize}[leftmargin=*]
	\item For $\xi=0.01$ (Figures \ref{fig:numerics} (a) and (b)) the interval $J_3$ lies to the right of $J_1$, so there might be a larger limit cycle further away from $J_1$.

    \item For $\xi=0.0125$ (so Figures \ref{fig:numerics} (c) and (d)) the interval $J_3$ intersects transversally $J_1$, and the intersection certifies the existence the singular periodic orbit.
    
    \item For $\xi=0.015$ (so Figures \ref{fig:numerics} (e) and (f)) the interval $J_3$ lies to the left of $J_1$, so there might be a smaller limit cycle further away from $J_1$, or the system might converge to the unique equilibrium point in the first octant.
\end{itemize}

It is worth noting that we chose to investigate the role of $\xi$, the birth/death rate, due to its biological relevance. However, by the same method one is able to numerically approach the existence of limit cycles upon variation of any other parameter. It is important to note that, in the limit systems, there is a clear separation between ``fast parameters'' ($\beta$, $\gamma$, $\nu$) and ``slow parameters'' ($\xi$, $\kappa$); changing a single parameter will only influence either the layer or the reduced dynamics, and not both.

\begin{figure}[htbp]
	\centering
	\begin{minipage}{\textwidth}
		\begin{subfigure}{.5\textwidth}
			\centering
			\includegraphics[scale=.85]{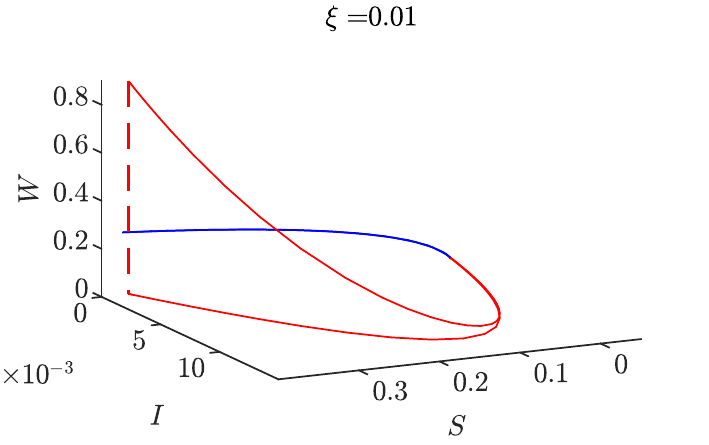}
			\caption{}
		\end{subfigure}
	\begin{subfigure}{0.5\textwidth}
		\centering
		\includegraphics[scale=.85]{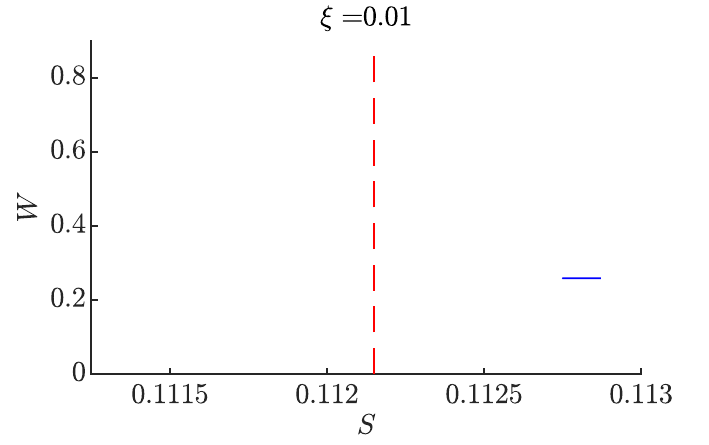}
		\caption{}
	\end{subfigure}%
	\end{minipage}\\
		\begin{minipage}{\textwidth}
		\begin{subfigure}{.5\textwidth}
			\centering
			\includegraphics[scale=.85]{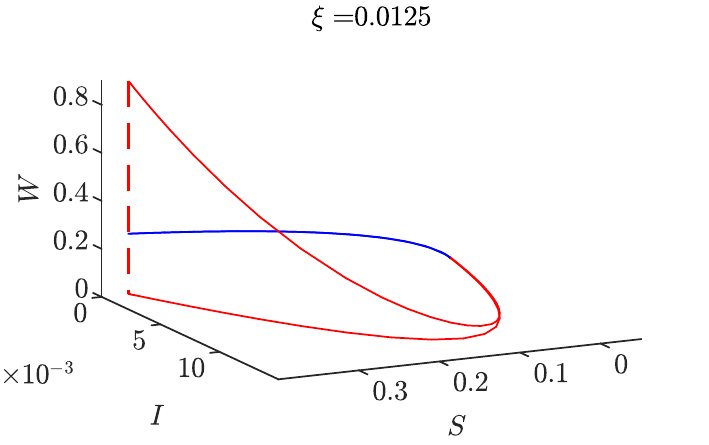}
			\caption{}
		\end{subfigure}
		\begin{subfigure}{0.5\textwidth}
			\centering
			\includegraphics[scale=.85]{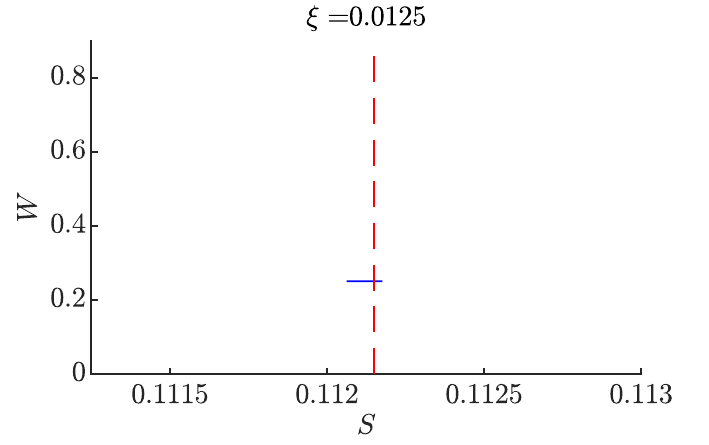}
			\caption{}
		\end{subfigure}%
	\end{minipage}\\
	\begin{minipage}{\textwidth}
	\begin{subfigure}{.5\textwidth}
		\centering
		\includegraphics[scale=.85]{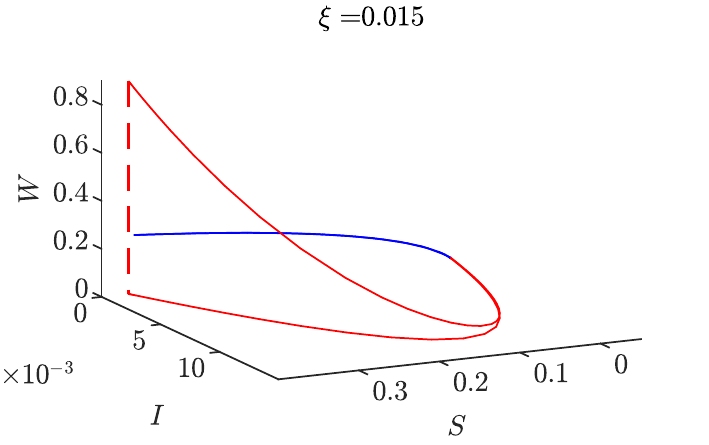}
		\caption{}
	\end{subfigure}
	\begin{subfigure}{0.5\textwidth}
		\centering
		\includegraphics[scale=.85]{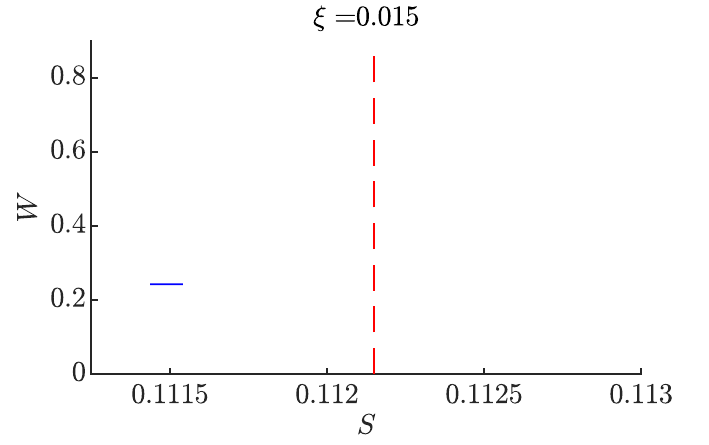}
		\caption{}
	\end{subfigure}%
\end{minipage}
\caption{Numerical illustration of the effect of changing $\xi$ on the slow dynamics. This numerical analysis shows that there is an interval around $\xi\sim0.0125$ for which periodic orbits of \eqref{eq:SIRWSslow} exists, for $\epsilon>0$ sufficiently small.}
\label{fig:numerics}
\end{figure}

Since we have already demonstrated the existence of limit cycles, the next question to investigate is the possible bifurcations that may arise upon variation of the parameters. Such analysis is presented in the forthcoming section.

\subsubsection{Bifurcation analysis}

In this section we carry out a bifurcation analysis, motivated by the one developed in \cite{Dafi}, which we perform with MatCont \cite{MatCont}. Our goal is to investigate the way the bifurcation diagrams change as $\epsilon$ is decreased, i.e., we want to understand via numerical continuation how the fast-slow singular limit is approached; see also~\cite{DesrochesKrauskopfOsinga1,GuckenheimerKuehn1,IuorioKuehnSzmolyan} where such a strategy has considerably improved our understanding of several fast-slow models. In our context, decreasing $\epsilon$ means, from a biological point of view, modelling an epidemiological system in which the difference in duration between life expectancy and infectious episodes becomes large. In the limit as $\epsilon\rightarrow 0$, infectious episodes become instantaneous, and the analysis of this limit case helps to understand the behaviour of the system for $\epsilon>0$ small enough.

In fact, we note that the system studied in \cite{Dafi} is system~(\ref{eq:SIRWSslow}), for the particular choice of $\epsilon=1$. In what follows, we set $\beta =260,\ \gamma =17,\ \kappa=0.1$, as in \cite{Dafi}, and vary $\epsilon$, $\xi$, $\nu$, and later $\beta$ as well.
Notice that the values of the parameters $\beta,\ \gamma,\ \kappa$ and $\xi$ already appear of different order of magnitude. It would be possible to use a different parametrization, letting $\tilde \beta = 0.26$, $\tilde \gamma = 0.017$ and $\epsilon = 0.001$. All the following analysis would be identical, except that the values obtained for $\epsilon$, $\beta$ and $\gamma$ would be multiplied by $10^{-3}$.

For consistency, we start by replicating Figure 5 from \cite{Dafi}, by setting $\epsilon=1$ and $\xi=0.01$, in Figure \ref{fig:periods}. For all parameter values there is a unique  equilibrium in $\mathbb R^4_{\geq0}$, as can be easily proved, but its stability changes varying $\nu$ through a subcritical and a supercritical Hopf bifurcation.

Next, in order to get the dependence of the bifurcation points with respect to $\epsilon$, we continue the two Hopf points $H_1$ and $H_2$ and the Limit Point of Cycles (LPC) $L$ in a $(\nu,\epsilon)$ bifurcation diagram, obtaining the diagram shown in Figure \ref{fig:bifurc}.

\begin{figure}
	\begin{subfigure}[t]{.45\textwidth}\centering
		\begin{tikzpicture}
		\node at (0,0){\includegraphics{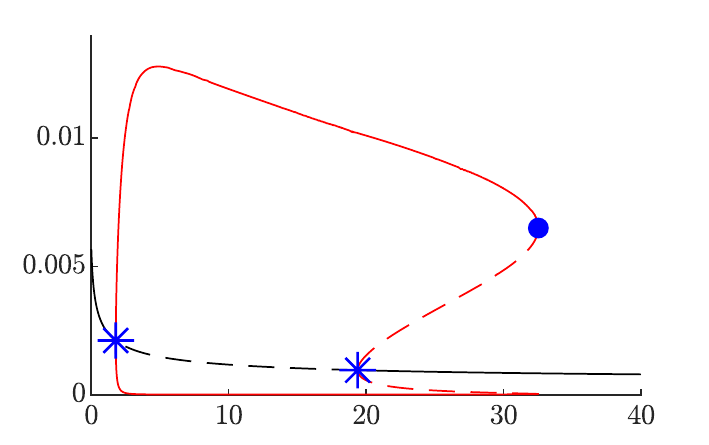}};
		\node at (-3.5,0){$I$};
		\node at (0,-2.5){$\nu$};
		\node[blue] at (-0.1,-1.1){$H_2$};
		\node[blue] at (-2,-1.1){$H_1$};
		\node[blue] at (2.2,0){$L$};
	\end{tikzpicture}
	\caption[]{One-parameter ($\nu$) bifurcation diagram for \eqref{eq:SIRWSslow}: blue stars labelled $H_1$ and $H_2$ correspond to Hopf points; blue dot labelled $L$ corresponds to the Limit Point of Cycles (LPC); red lines correspond to stable (solid) and unstable (dashed) limit cycles; the stable (solid) and unstable (dashed) equilibrium point is depicted by the black line.}
	\label{fig:periods}
	\end{subfigure}\hfill
	\begin{subfigure}[t]{.45\textwidth}\centering
		\begin{tikzpicture}
		\node at (0,0){\includegraphics[scale=1.]{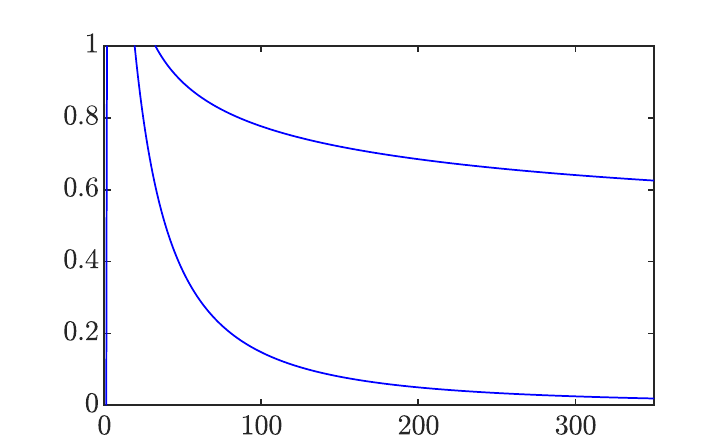}};
		\node at (-3.25,0){$\epsilon$};
		\node at (0,-2.5){$\nu$};
		\node at (-1.8,.7){\small $H_2$};
		\node at (-2.2,-1){\small $H_1$};	
		\node at (-.7,1.1){\small $L$};
	\end{tikzpicture}	
		\caption[]{The blue lines represent the Hopf points $H_1$ and $H_2$, and the LPC point $L$, plotted in Figure~(\ref{fig:periods}), which are then continued while decreasing $\epsilon$; compare with Figure \ref{fig:periods}. We observe that $H_1$ does not tend to $\nu=0$ as $\epsilon\to0$ while $H_2$ and $L$ diverge.}
		\label{fig:bifurc}
	\end{subfigure}
	\caption{One and two parameter bifurcation diagrams for \eqref{eq:SIRWSslow}.}
	\label{fig:bifs1}
\end{figure}
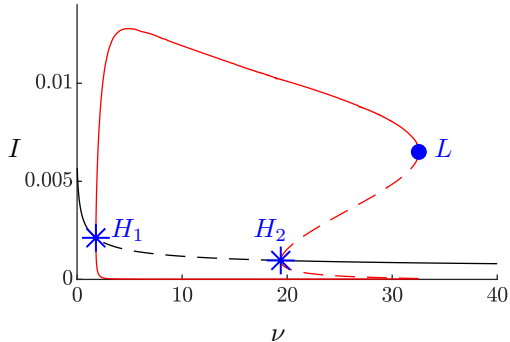
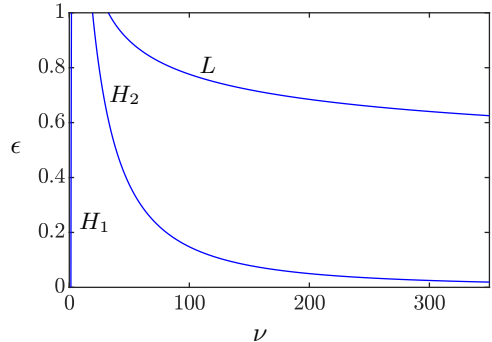

We notice from Figure \ref{fig:bifurc} that $H_1$ converges to a positive value for $\nu\sim 1.32$ as $\epsilon \rightarrow 0$, while $H_2$ and $L$ diverge; the latter much faster than the former. Moreover, we know from the analysis performed in Section \ref{sec:SIRWS} that as $\epsilon\to0$ the equilibrium curve (black curve in Figure \ref{fig:periods}) approaches the $\left\{ I=0 \right\}$ axis. These two observations suggest that as $\epsilon\to0$ the bifurcation diagram on Figure \ref{fig:periods} gets stretched. One must also point out that the computation of the bifurcation diagrams for small $\epsilon$ becomes considerably expensive due to the high stiffness of the problem.

We next produce the analogous to Figure~\ref{fig:periods}, but for a smaller value of $\epsilon$, namely $\epsilon=0.05$, in Figure~\ref{fig:per_005}. In order to do so, due to stiffness of the problem, it is necessary to rescale the system by introducing a new variable $v=\ln(I)$. We emphasize that this rescaling is motivated by the fact that trajectories get exponentially close to the critical manifold, recall Lemma \ref{lemm:l}. Moreover, this rescaling might be useful for bifurcation analysis of systems with similar dynamics in which an exchange of stability of the critical manifold occur at a non-hyperbolic point, and trajectories of interest pass exponentially close to such a singularity. With the aforementioned rescaling one obtains the following system of ODEs:
\begin{equation}\label{eq:SIRWSlog}
\begin{aligned}
S'&= -\beta S\e^v+\epsilon (2\kappa W +\xi (1-S)),\\
v'&=v(\beta S-\gamma -\epsilon  \xi),\\
W'&=-\nu\beta W\e^v +\epsilon (2\kappa (1-S-\e^v-W)-2\kappa W-\xi W).
\end{aligned}
\end{equation}

\begin{figure}[htbp]
	\centering
	\begin{tikzpicture}
      \node at (0,0){\includegraphics[scale=1.15]{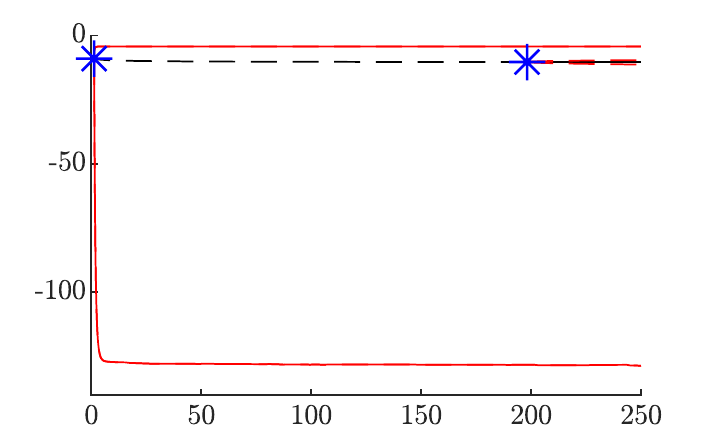}};
		\node at (-3.75,0){$v$};
		\node at (0,-2.75){$\nu$};
		\node[blue] at (1.7,1.5){$H_2$};
		\node[blue] at (-2.7,1.5){$H_1$};
		\end{tikzpicture}	
		\caption{One-parameter ($\nu$) bifurcation diagram for \eqref{eq:SIRWSlog}: blue stars labelled $H_1$ and $H_2$ correspond to Hopf points; red lines correspond to stable (solid) and unstable (dashed) limit cycles; the stable (solid) and unstable (dashed) equilibrium point is depicted by the black line.}
				\label{fig:per_005}
\end{figure}
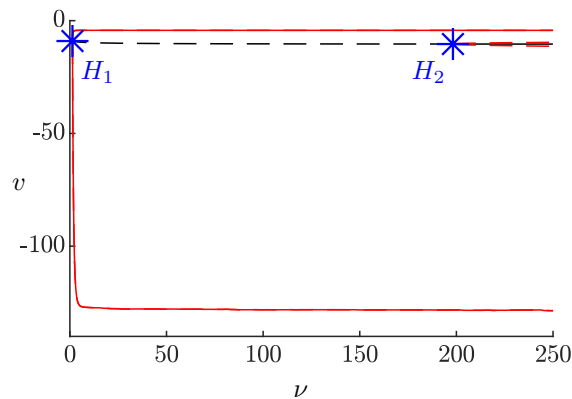

Thus, the bifurcation diagram in Figure~\ref{fig:per_005} is obtained from \eqref{eq:SIRWSlog} and confirms the behaviour anticipated in Figure~\ref{fig:bifurc}: as $\epsilon$ decreases, the distance between $H_1$ and $H_2$ increases, thus stretching the parameter region in which stable periodic solutions are to be observed. Most importantly, as is already evident in Figure \ref{fig:bifurc}, we have that for $\epsilon$ sufficiently small the LPC is undetectable, implying that an eventual transition to stable (endemic) equilibrium due to increase of the immunity boosting rate $\nu$ is not possible any more. 

Another important parameter is $\beta$, which regulates the infection rate. Thus, in order to further investigate the role of $\epsilon$ in the model, we next present in Figure \ref{fig:betanu_bifurc} a $(\nu,\beta)$ bifurcation diagram.

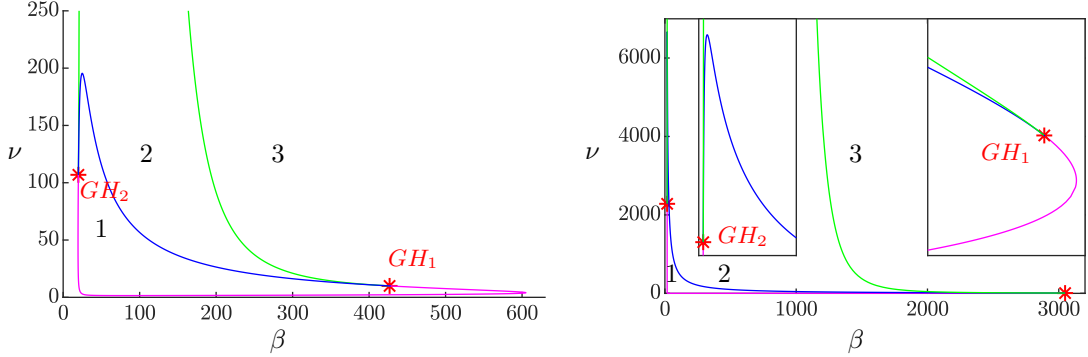
\begin{figure}[H]\centering
	\begin{tikzpicture}
	\node at (0,0){\includegraphics{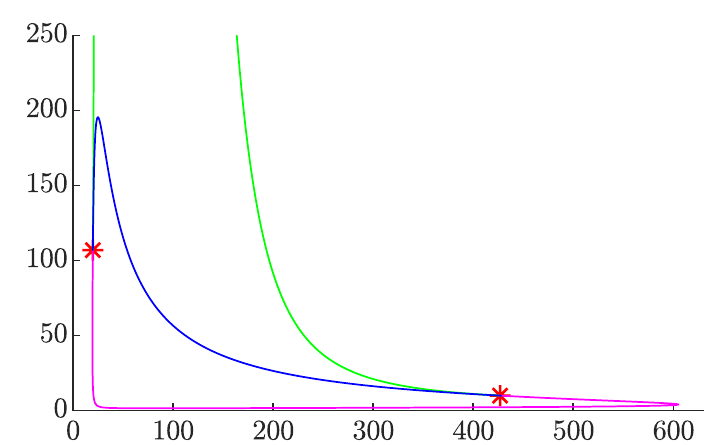}};
	\node at (-3.5,0){$\nu$};
	\node at (0,-2.5){$\beta$};
	\node at (-2.35,-1){1};
	\node at (-1.75,0){2};
	\node at (0,0){3};
	\node[red] at (1.8,-1.4){\small $GH_1$};
	\node[red] at (-2.3,-0.5){\small $GH_2$};
	\end{tikzpicture}\hfill
	\begin{tikzpicture}
	\node at (0,0){\includegraphics{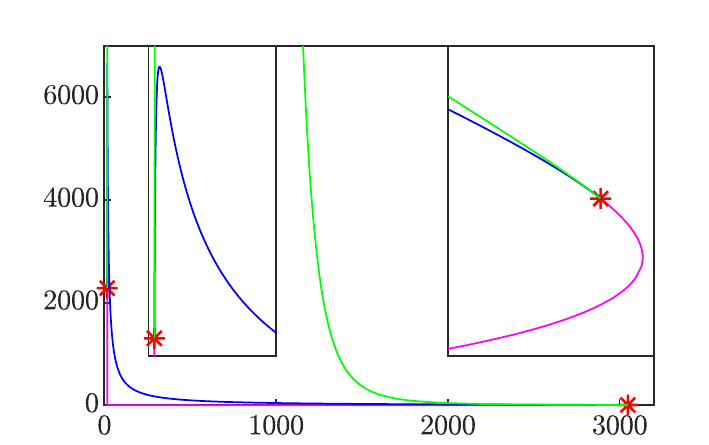}};
	\node at (-3.5,0){$\nu$};
	\node at (0,-2.5){$\beta$};
	\node at (-2.45,-1.6){1};
	\node at (-1.75,-1.6){2};
	\node at (0,0){3};
	\node[red] at (2,0){\small $GH_1$};
	\node[red] at (-1.5,-1.1){\small $GH_2$};
	\end{tikzpicture}
	\caption{Two parameter bifurcation diagram for \eqref{eq:SIRWSslow}. Left and right represent $\epsilon=1$ and $\epsilon=0.05$, respectively. The red points labelled $GH_i$ are generalised Hopf points. The blue (resp. magenta) branch is a curve of subcritical (resp. supercritical) Hopf bifurcation while the green branches correspond to limit point of cycles. We label the regions in the diagram according to the attractor as 1: Limit cycles, 2: Bistability, and 3: Point attractor. The insets in the right picture are ``zoom-ins'' near the two $GH$ points.} 
	\label{fig:betanu_bifurc} 
\end{figure}

For ease of notation, let us denote by $\nu(P)$ the value of $\nu$ corresponding to a point $P$. From Figure \ref{fig:betanu_bifurc} we have that $\nu(GH_1) \approx 9.96$ and $\nu(GH_2) \approx 106.9$ for $\epsilon=1$. Furthermore, for $\nu \leq \nu(GH_1)$, the system only exhibits stability of the equilibrium or of the limit cycle (zones $1$ and $3$). For $\nu(GH_1)< \nu \leq \nu(GH_2)$ there are two intervals of values for $\beta$ which correspond to a stable equilibrium, one to a stable limit cycle and one to bistability (zones $1$, $2$, and $3$). For $\nu(GH_2) < \nu \leq \nu_{\max}$, with $\nu_{\max} \approx 195.46$, there are two intervals of values for $\beta$ which correspond to a stable equilibrium, one to a stable limit cycle and two to bistability, one of them being very thin. At $\nu = \nu_{\max}$  the two Hopf points $H_1$ and $H_2$ collide, and a codimension-2 Hopf-Hopf bifurcation occurs.

For $\epsilon=0.05$, the diagram is qualitatively the same, but as already pointed-out before the diagram gets stretched both in $\beta$ and in $\nu$. The points $GH_1$ and $GH_2$ correspond now to $\nu \approx 7.04$ and $\nu \approx 2282.6$, respectively. In particular, the bistability region $2$ is enlarged.

To complement the previous description, and similar to Figure 9 (a) to (d) in \cite{Dafi} in Figures~\ref{fig:nu6}-\ref{fig:nu190}, we present the $\beta$-bifurcation diagram for different values of $\nu$ and continue all the Hopf points for decreasing $\epsilon$, as shown in Figures~\ref{fig:nu6cont}-\ref{fig:nu190cont}.

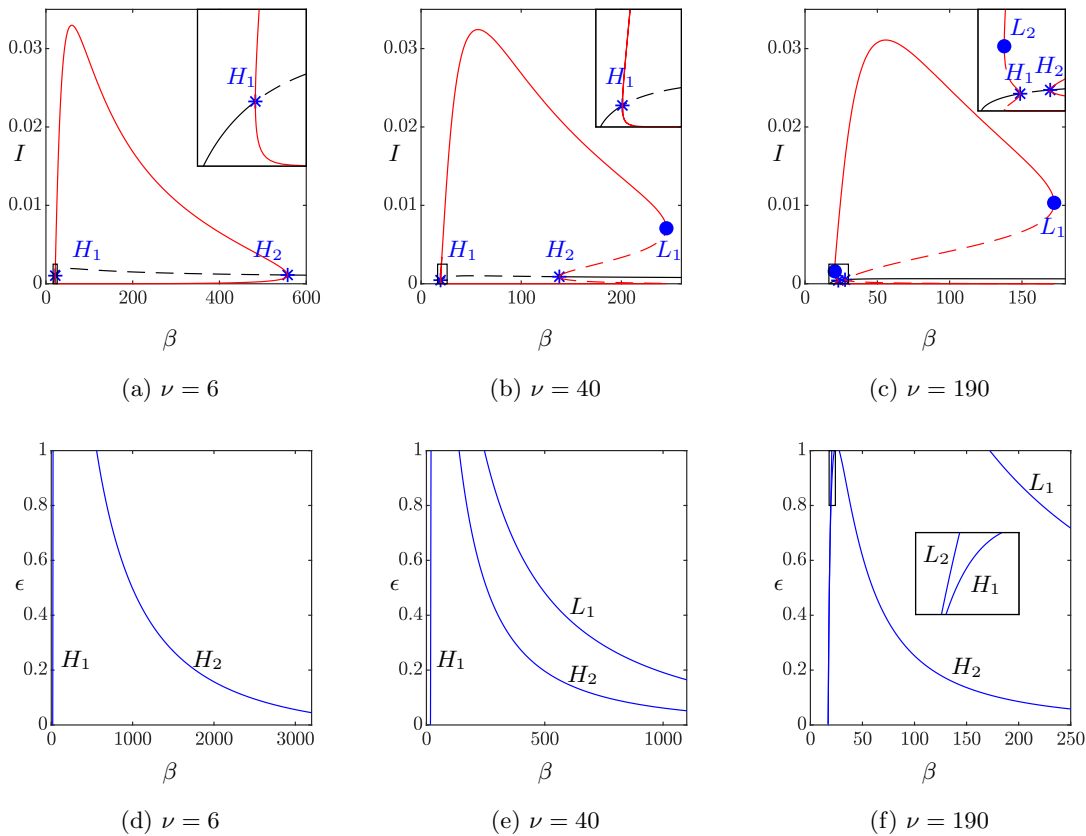
\begin{figure}[htbp] 
	\begin{subfigure}[b]{0.33\linewidth}
		\centering
			\begin{tikzpicture}
		\node at (0,0){\includegraphics[width=.9\linewidth]{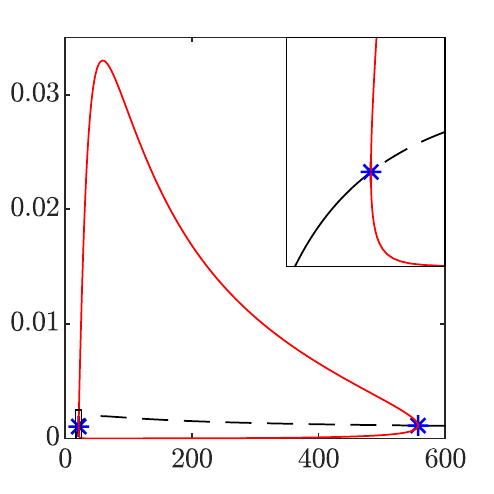}};
		\node at (0,-2.5){\small $\beta$};
		\node at (-2,0){\small $I$};
		\node[blue] at (-1.1,-1.3){\small $H_1$};
		\node[blue] at (1.3,-1.3){\small $H_2$};
		\node[blue] at (0.95,1){\small $H_1$};
				\end{tikzpicture} 
		\caption{$\nu=6$} 
		\label{fig:nu6} 
		\end{subfigure}
	\begin{subfigure}[b]{0.33\linewidth}
		\centering
		\begin{tikzpicture}
		\node at (0,0){\includegraphics[width=.9\linewidth]{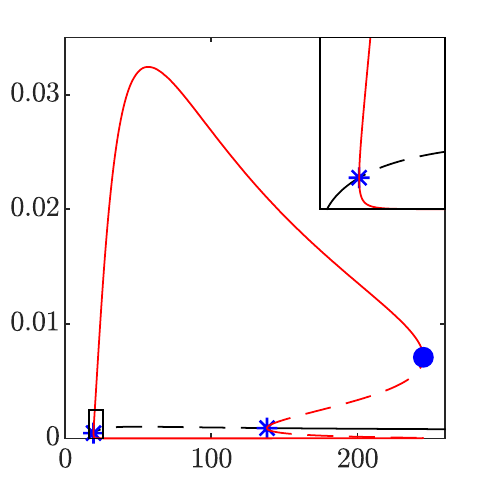}};
		\node at (0,-2.5){\small $\beta$};
		\node at (-2,0){\small $I$};
		\node[blue] at (-1.1,-1.3){\small $H_1$};
		\node[blue] at (1,1){\small $H_1$};
		\node[blue] at (0.2,-1.3){\small $H_2$};
		\node[blue] at (1.65,-1.3){\small $L_1$};
		\end{tikzpicture}  
		\caption{$\nu=40$} 
		\label{fig:nu40} 
		\end{subfigure} 
	\begin{subfigure}[b]{0.33\linewidth}
		\centering
		\begin{tikzpicture}
		\node at (0,0){\includegraphics[width=.9\linewidth]{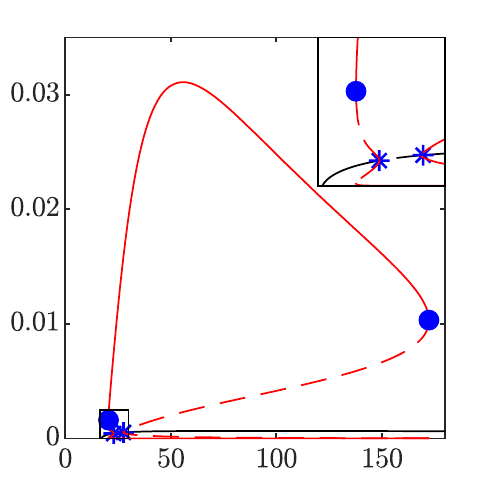}};
		\node at (0,-2.5){\small $\beta$};
		\node at (-2,0){\small $I$};
		\node[blue] at (1.2,1.05){\small $H_1$};
		\node[blue] at (1.6,1.15){\small $H_2$};
		\node[blue] at (1.25,1.65){\small $L_2$};
		\node[blue] at (1.65,-1){\small $L_1$};
		\end{tikzpicture}  
		\caption{$\nu=190$} 
		\label{fig:nu190} 
		\end{subfigure}\\
	\begin{subfigure}[b]{0.33\linewidth}
	\centering
	\begin{tikzpicture}
	\node at (0,0){\includegraphics[width=.9\linewidth]{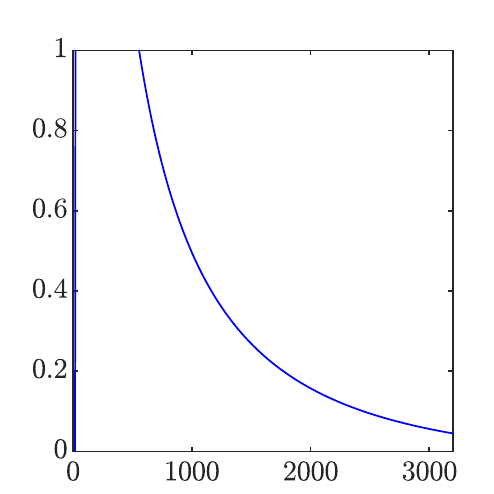}};
	\node at (0,-2.5){\small $\beta$};
	\node at (-2,0){\small $\epsilon$};
	\node at (-1.25,-1){\small$H_1$};
	\node at ( .5,-1){\small$H_2$};
	\end{tikzpicture}   
	\caption{$\nu=6$} 
	\label{fig:nu6cont} 
	\end{subfigure}
\begin{subfigure}[b]{0.33\linewidth}
	\centering
	\begin{tikzpicture}
	\node at (0,0){\includegraphics[width=.9\linewidth]{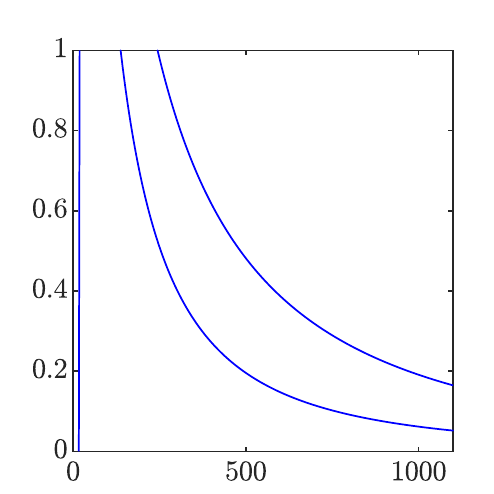}};
	\node at (0,-2.5){\small $\beta$};
	\node at (-2,0){\small $\epsilon$};
	\node at (-1.25,-1){\small$H_1$};
	\node at ( .5,-1.2){\small$H_2$};
	\node at ( .5,-.3){\small$L_1$};
	\end{tikzpicture}    
	\caption{$\nu=40$} 
	\label{fig:nu40cont}
\end{subfigure} 
\begin{subfigure}[b]{0.33\linewidth}
	\centering
	\begin{tikzpicture}
	\node at (0,0){\includegraphics[width=.9\linewidth]{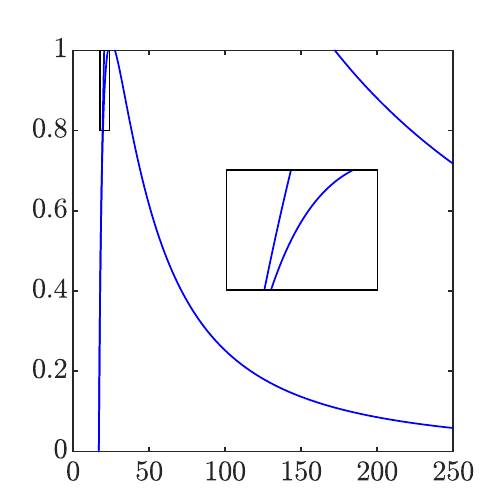}};
	\node at (0,-2.5){\small $\beta$};
	\node at (-2,0){\small $\epsilon$};
	\node at ( .5,-1.1){\small$H_2$};
	\node at ( .75,0){\small$H_1$};
	\node at ( 0.1,0.4){\small$L_2$};
	\node at ( 1.5,1.35){\small$L_1$};
	\end{tikzpicture}    
	\caption{$\nu=190$}
	\label{fig:nu190cont} 
\end{subfigure}
\caption[]{{\textbf{First row:}} one-parameter ($\beta$) bifurcation diagram for \eqref{eq:SIRWSslow}: blue stars labelled $H_1$ and $H_2$ correspond to Hopf points; blue circles labelled $L_1$ and $L_2$ correspond to Limit Point of Cycles; red lines correspond to stable (solid) and unstable (dashed) limit cycles; the stable (solid) and unstable (dashed) equilibrium point is depicted by the black line. The insets correspond to zoom-in near $\beta=17$. {\textbf{Second row:}} continuation of the Hopf and LPC points while decreasing $\epsilon$. We observe that $H_1$ (and $L_2$, when it exists) tends to $\beta=17$ as $\epsilon\to0$, while $H_2$ (and $L_1$, when it exists) diverges. The inset in (f) shows a zoom-in at the continuation of $H_1$ and $L_2$ from $\epsilon=1$ to $\epsilon=0.8$.}
\end{figure}

As before, and for ease of notation, we denote by $\beta(P)$ the value of $\beta$ corresponding to a point $P$. For each value of $\nu$ considered, we find two values $17<\beta(H_1)<\beta(H_2)$ ($17$ was the fixed value of $\gamma$ in each simulation; recall $R_0=\beta/\gamma$) corresponding to Hopf points, and we continue them in $\epsilon$, as shown in Figures~\ref{fig:nu6cont}-\ref{fig:nu190cont}. For $17 \leq \beta \leq \beta(H_1)$ the equilibrium point is stable, and there is no limit cycle. For $\beta(H_1) < \beta \leq \beta(H_2)$ the equilibrium point is unstable, and the limit cycle stable. For $\nu>\nu(GH_1)$ (resp. $\nu>\nu(GH_2)$), there is an interval (resp. there are two intervals) of values of $\beta(H_2)<\beta\leq\beta(L)$ (with $L$ a LPC, whose existence and position depend on the choice of $\nu$) for which the system exhibits bistability; eventually these two limit cycles collapse, and for $\beta>\beta(L)$ the system is characterized by a unique asymptotically stable equilibrium. Note, interestingly, that as the Hopf-Hopf bifurcation is approached, a new LPC ($L_2$ in Figure \ref{fig:nu190}) becomes visible.

We note that in the limit $\epsilon \rightarrow 0$, one has $\beta(H_1)\to 17$. This is due to the influence on the dynamics of the basic reproduction number $R_0 = \beta /\gamma$, which should remain greater than $1$ for the endemic equilibrium to exist. Related to this, one has that $\beta(L_2)\to 17$ as $\epsilon \to 0$, whenever $\nu>\nu(GH_2)$. The values $\beta(H_2)$ and $\beta(L_1)$, instead, diverge to $+\infty$ as $\epsilon \rightarrow 0$; the region corresponding to the stable limit cycle stretches, as in the $\nu$ case. Lastly, we compute a $(\xi,\nu)$-diagram and compare them for $\epsilon=1$ and $\epsilon=0.05$ in Figure \ref{fig:two_par}, as we did for $(\beta,\nu)$ in Figure \ref{fig:betanu_bifurc}.
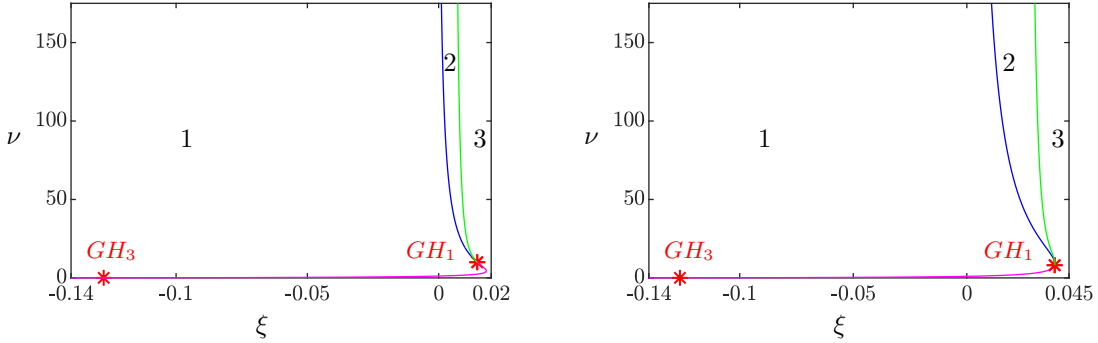
\begin{figure}[H]\centering
	\begin{tikzpicture}
	\node at (0,0){\includegraphics[scale=1]{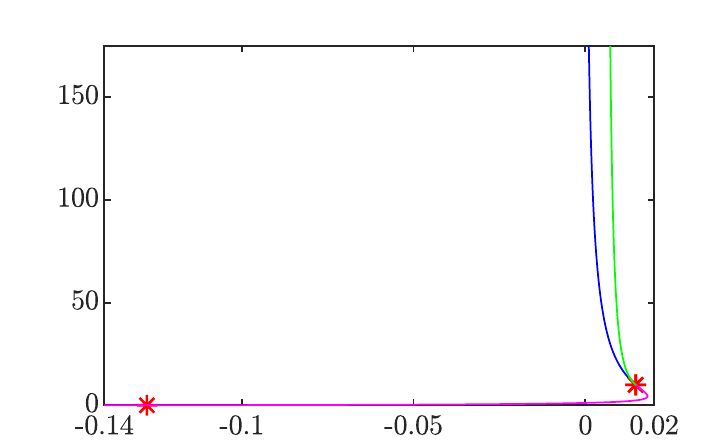}};
	\node at (-3.3,0){$\nu$};
	\node at (0,-2.5){$\xi$};
	\node at (-1,0){1};
	\node at (2.5,1){2};
	\node at (2.9,0){3};
	\node[red] at (2.25,-1.5){\small $GH_1$};
	\node[red] at (-2,-1.5){\small $GH_3$};
	\end{tikzpicture}\hfill
	\begin{tikzpicture}
	\node at (0,0){\includegraphics[scale=1]{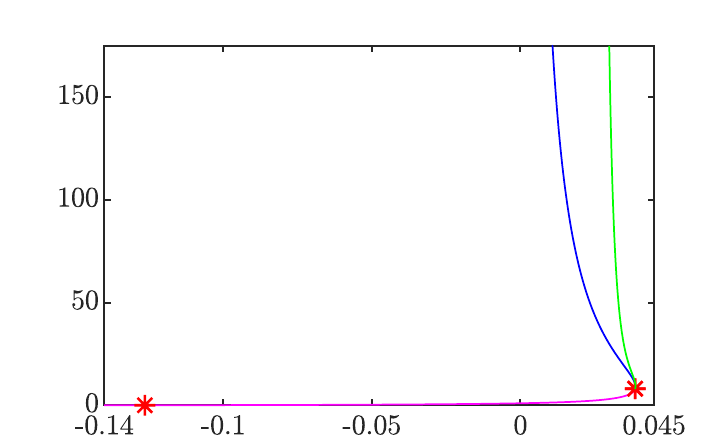}};
	\node at (-3.3,0){$\nu$};
	\node at (0,-2.5){$\xi$};
	\node at (-1,0){1};
	\node at (2.25,1){2};
	\node at (2.9,0){3};
	\node[red] at (2.25,-1.5){\small $GH_1$};
	\node[red] at (-2,-1.5){\small $GH_3$};
	\end{tikzpicture}
		\caption{Two parameter bifurcation diagram for \eqref{eq:SIRWSslow}. Left and right represent $\epsilon=1$ and $\epsilon=0.05$, respectively. The red points labelled $GH_i$ are generalised Hopf points. The blue (resp. magenta) branch is a curve of subcritical (resp. supercritical) Hopf bifurcation while the green branch corresponds to a limit point of cycles. Thus, we label the regions in the diagram according to the attractor as 1: Limit cycles, 2: Bistability, and 3: Point attractor.}	
		\label{fig:two_par}
\end{figure}

We observe in Figure \ref{fig:two_par} that not only the bifurcation diagram is stretched as $\epsilon$ decreases but also that the bistable region (region 2) is enlarged. $GH_1$ corresponds to $\xi \approx 0.0147$ for $\epsilon=1$ and to $\xi \approx 0.03871$ for $\epsilon=0.05$. Furthermore, in Figure~\ref{fig:two_par} we show the existence of another Generalized Hopf point $GH_3$ (not considered in \cite{Dafi}), corresponding to $\xi \approx -0.1276$ for $\epsilon=1$ and to $\xi \approx -0.1263$ for $\epsilon=0.05$. We do not show the $2$-parameter continuation of $GH_3$ since such a computation is not numerically feasible due to the high stiffness of the system in such parameter range. However, the previous observation suggests that all the bifurcation branches corresponding to $GH_3$ are close to each other.

The numerical analysis shown in this section supports the existence of stable limit cycles for an increasing parameter range as $\epsilon \to 0$. Nonetheless, the dependence of the behaviour of the orbits on the parameters stays the same for sufficiently small parameters. This means that as in the $\epsilon=1$ case, one still observes parameter ranges corresponding to the stability of the endemic equilibrium, and other parameter ranges corresponding to stable periodic orbits.

Based on the analysis performed so far, we can now give an interpretation of our results: first of all, the interplay between birth/death rate $\xi$ and immune boosting $\nu$ remains qualitatively similar to the one described in \cite{Dafi}, for small $\epsilon$. However, the Hopf point $H_2$ moves according to the increasing difference in the time scales involved in the respective dynamics. $H_1$ does not converge to $0$, supporting the result obtained in \cite{Dafi}, where the authors showed that, for $\nu$ small enough, the dynamics are close to a SIRS system. The main difference, however, is that as $\epsilon$ decreases the role of the parameters can drastically change due to the changes in the bifurcation diagram. For example, for $\epsilon=1$, a life expectancy of $50$ years ($\xi =0.02$) corresponds to convergence to the endemic equilibrium for all the possible values of $\nu$. In contrast, for smaller values of $\epsilon$ the same $\xi$ could correspond to stability of the limit cycle, bistability, or stability of the endemic equilibrium, depending on the value of $\nu$ (see Figure \ref{fig:two_par}). Moreover, the effect of increasing life expectancy, i.e. decreasing $\xi$, results in the transition from point stability to stability of a limit cycle, possibly passing through a region of bistability. This means that, the higher the life expectancy of a certain population, the larger the interval for $\nu$ for which a stable limit cycle exists. Biologically, this means that $\nu$ must be sufficiently small to obtain a stable endemic equilibrium, otherwise periodic epidemic outbursts turn out to be robust. 

\section{Summary and Outlook}\label{sec:conclusions}

We have analysed the behaviour of three models given as a nonstandard singularly perturbed ODE. The first two models presented in Sections \ref{SIR} and \ref{SIRS} proved to behave, under mild hypotheses on the parameters, qualitatively in the same way. In particular, their trajectories converge to the only (endemic) equilibrium in the open first quadrant, as long as the initial population of infected individuals is strictly positive. The SIRWS model, instead, proved to be much richer, with parameter regimes allowing for damped oscillations or sustained oscillations, or both.

For our analysis we have combined techniques from Geometric Singular Perturbation Theory, and in particular the entry-exit function, introduced in section \ref{entryexit}. One must point-out that GSPT is usually employed for singular perturbation problems in standard form, and just recently it has been shown that non-standard problems can also be dealt with. More precisely, GSPT allowed us to show the existence of stable limit cycles for certain parameter ranges. Based on such analysis, we further performed numerical studies and computed several insightful bifurcation diagrams, which allowed us to provide a complete qualitative description of the perturbed SIRWS model.

We concluded comparing previous results appearing in \cite{Dafi}, and extending them by taking into account the role of the (small) parameter $\epsilon$, which does not change the overall qualitatively behaviour of the system, but it does drastically change the parameter ranges corresponding to each dynamic regime. Finally, our studies show that GSPT together with numerical tools seem to be suitable to analyze and comprehend epidemiological models with vastly different rates.

Once the bifurcation structure of epidemic models is known, one can then be more ambitious and aim to not only control epidemic outbreaks better after they have occurred but even try to anticipate them using early-warning signs~\cite{OReganDrake,WidderKuehn}. Therefore, our results on bifurcation structure presented here are strongly expected to contribute to the design of these warning signs.\medskip

\textbf{Acknowledgments:} HJK would like to thank the Alexander-von-Humboldt Foundation for funding via a fellowship. CK would like to thank the VolkswagenStiftung for support via a Lichtenberg Professorship. MS would like to thank the University of Trento for supporting his research stay at the Technical University Munich. AP thanks Barbara Boldin and Odo Diekmann for useful discussions that started the interest in this project, and for sharing the computations they made; Nico Stollenwerk for useful discussion about similar but more complex models.

\bibliographystyle{plain} \small
\bibliography{biblio}

\end{document}